\documentclass[reqno,12pt]{amsart}
\usepackage{ragged2e}
\usepackage{amssymb}
\usepackage{amsfonts}
\usepackage{amsmath}
\usepackage{mathrsfs}
\usepackage{diagbox}
\usepackage{color,soul}
\usepackage{bbm}
\usepackage{hyperref}
\usepackage{amsthm, graphicx,empheq}
\usepackage{framed}
\newtheorem{lemma}{Lemma}[section]
\newtheorem{theorem}{Theorem}[section]
\newtheorem{definition}{Definition}[section]
\newtheorem{remark}{Remark}[section]
\numberwithin{equation}{section}

\makeatletter
\@namedef{subjclassname@2020}{\textup{2020} Mathematics Subject Classification}
\makeatother
\allowdisplaybreaks[4]

\begin{document}

\title[String Vibration Equations with Elastic Supports]{Mathematical Analysis and Numerical Computation of String Vibration Equations with Elastic Supports for Bridge Cable Force Evaluation}

\author{MINHUI TAN}
\author{QING XU}
\author{HAIRONG YUAN}
\author{MAN XU}
\author{KE LIU}
\author{AIFANG QU}
\author{XIAODA XU}
\address[M. Tan]{Department of Mathematics, Shanghai Normal University,
Shanghai, 200234, China}\email{\tt 1000465950@smail.shnu.edu.cn}

\address[Q. Xu]{Central Research  Institute of Building and Construction Co., Ltd. MCC Group, Beijing 100088, China \& Department of Civil Engineering, Tsinghua University, Beijing 100084, China}\email{\tt a\_qing01@foxmail.com}

\address[H. Yuan]{School of Mathematical Sciences, Key Laboratory of MEA (Ministry of Education) \& Shanghai Key Laboratory of PMMP, East China Normal University, Shanghai 200241, China}\email{\tt hryuan@math.ecnu.edu.cn}

\address[M. Xu]{Central Research  Institute of Building and Construction Co., Ltd.MCC Group, Beijing 100088, China}\email{\tt manman3112899@126.com}

\address[K. Liu]{School of Mathematical Sciences, Key Laboratory of MEA (Ministry of Education) \& Shanghai Key Laboratory of PMMP, East China Normal University, Shanghai 200241, China}\email{\tt 52275500033@stu.encu.edu.cn}

\address[A. Qu]{Department of Mathematics, Shanghai Normal University,
Shanghai, 200234, China}\email{\tt afqu@shnu.edu.cn}

\address[X. Xu]{Central Research  Institute of Building and Construction Co., Ltd.MCC Group, Beijing 100088, China}\email{\tt 13115279@bjtu.edu.cn}


%msc2020

\date{\today}

\begin{abstract}
This study focuses on a critical aspect of bridge engineering---the evaluation of cable forces, paying particular attention to the cables that are internally constrained by elastic supports. Detecting these cable forces is important for the safety and stability of bridges. The practical problem introduces a novel mathematical challenge: how to effectively address string vibration equations with one or multiple internal elastic supports, which remains a theoretical issue not fully solved in engineering. To tackle this, it is necessary to firstly establish an appropriate mathematical model and accurately define initial-boundary value problems. We then formulate the well-posedness of the solution using both classical and weak solution approaches, supplementing the existing numerical results available in engineering. Meanwhile, we attempt to use PINNs (Physics-Informed Neural Networks) instead of traditional FEM (Finite Element Method) in engineering. Consequently, in contrast to the classical solution method, we demonstrate that for a string with finite elastic supports, the weak solution method not only improves mathematical modeling efficiency but also simplifies the process of explaining the well-posedness of the solution.
\end{abstract}

\allowbreak
\allowdisplaybreaks

\maketitle

\tableofcontents %disable for short paper

\section{Introduction}\label{sec1}
\par
In this paper, we would like to study string vibration equations with a finite number of elastic supports, and the initial-boundary value problem is demonstrated as
\begin{equation}\label{eq1.1}
\left\{\begin{array}{l}
u_{t t}-a^{2} u_{x x}=\sum\limits_{k=1}^{n}-\beta_k u\left(x_{k}, t\right) \delta_{x=x_k},  (x, t) \in(0, L) \times(0, T], \\
u(0,t)=0, \quad u(L,t)=0,  t\in(0, T], \\
u(x,0)=\varphi(x), \quad u_{t}(x,0)=\psi(x),  x\in[0, L],
\end{array}\right.
\end{equation}
where $L$ and $\beta_k$ are positive constants and $x_k\in(0,L) (k=1,2,\dots,n)$. We point out that in \eqref{eq1.1}, only the independent variables $x$ and $t$ are included, in which $x$ represents position, $t$ represents time, and $a$ is a positive constant, $\varphi(x)$, $\psi(x)$ are given functions while $u(x, t)$ is a unknown function and $\delta_{x=x_k}$ is the so called Dirac measure supported at $x=x_k$.
\par
The problem \eqref{eq1.1} stems from a key issue in bridge engineering. It is well known that prestressed cables play a crucial role in these large-span structures because these special materials can not only support the weight but also help to control the overall sturdiness of the structures. However, during the usage, many factors such as continuous heavy pressure, corrosion or fatigue, can redistribute the internal forces within the structures. Therefore, detecting cable tension is important for ensuring the construction quality and safety, as well as for routine maintenance and monitoring.
\par
The current methods for measuring cable tension include pressure sensor testing, magnetic flux testing, frequency analysis, etc.. Among these, frequency analysis is one of the most practical and convenient  method. There are two general engineering models for frequency analysis called the taut-string model and the beam model \cite{MT1998,CWCL2018,GFD2021,FW2012} which are both based on the assumptions that prestressed cables are directly connected to the structure at both ends and without internally-constrained supports. Nevertheless, in practical engineering, to effectively control the vibration of prestressed cables, the middle part of the cables is often equipped with connectors, such as cable clamps, dampers and support frames, thus giving the cables a multi-support characteristic. As far as we know, the method of ordinary differential equations is used to describe the boundary value problems of cables and beams with concentrated loads,  and the importance of such problems is emphasized \cite{ICA1989}. However, there is no reference that starts from the basic vibration equation of an internally constrained cable and explores its difference from the typical vibration equations \cite{GF2015}.
\par
Specifically, we focus on the motion and vibration frequency of prestressed cables, disregarding the natural vibrations of the struts on these cables. Thus, in current engineering practices, the tensioned cables of bridges are often idealized as strings, and the struts on these cables are simplified to elastic supports with motion only in the vertical direction, providing concentrated external forces. All these modelings lead to string vibrations equations with a finite number of elastic supports, i.e. the problem \eqref{eq1.1}.
\par
As seen in \eqref{eq1.1}, $u(x,t)$ is the displacement of each point on the string in the direction perpendicular to the $x$-axis. With the total length of the string being $L$, the ends of the string are fixed at $x=0$ and $x=L$, and the elastic support is placed at $x=x_k (0<x_k<L)$. It is also assumed that the initial position and velocity of the string at time $t=0$ are $\varphi(x)$ and $\psi(x)$ repectively. Moreover, $a^2=\frac{T}{\rho}$, $\beta_{k}=\frac{\mathsf{K}_{k}}{\rho}$, where $\rho$ represents the linear density of the cable material, $T$ is the cable force, ${\mathsf{K}_k}$ is the stiffness coefficient corresponding to each elastic support at $x=x_k$.
\par
This specific practical problem \eqref{eq1.1} remains an unresolved theoretical issue in the engineering, leading to a new mathematical problem---how to solve an initial-boundary value problem with a second-order linear hyperbolic equation containing point-source terms characterizing as external forces and so on. In fact, as far as we know, previous results such as \cite{CR1962}, have not considered the solution of an initial-boundary value problem for string vibration equations with Dirac measures as external force terms, as well as how to discuss the well-posedness of the solution. Thus, regular methods like Fourier transform cannot be directly applied. Therefore, this research aims to provide rigorous physical modeling methods, thereby exploring a new class of second-order hyperbolic equations including Dirac measures. Also, we would like to present theoretical foundations for engineering practices and assumptions. In this work, we approach from a rigorous mathematical perspective, focusing directly on the definition and well-posedness of the solution to \eqref{eq1.1}.
\par
In fact, to explore this bridge engineering problem, we initially seek alternatives to the physical model \eqref{eq1.1} in the beginning of our research. Starting from an intuitive understanding of the physical form, we reconstruct a new physical model without point-source external force, present the corresponding initial-boundary value problems, and achieve both the formal analytical expression and the well-posedness of the solution. Simultaneously, we attempt to use PINNs (Physics-Informed Neural Networks) and observe stabilizing effects of elastic supports to some extent, in accordance with engineering requirements and conclusions. Additionally, we compare the theoretical values of string vibration frequencies obtained from the physical model in the sense of classical solution with corresponding numerical results brought by FEM (Finite Element Method) in engineering with a minor discrepancy rate. Ultimately, by comparing these two different physical models, we find that the weak solution approach corresponding to \eqref{eq1.1} is superior in handling the problem with a finite number of elastic supports than the classical solution approach with a higher efficiency of physical modeling and discussion for the properties of solution.
\par
The following content of this paper is divided into three sections. In the second section, by observing the intuitive physical model, we propose a new coupled problem, which is completely different from \eqref{eq1.1}, to formulate the string vibration equation with elastic supports. Under this circumstance, we obtain the theoretical solution involving implicit functions, illustrate the well-posedness of the solution, and seek to fit this intuitive physical model by using PINNs. In the third section, we directly conduct a study of the physical model \eqref{eq1.1} as proposed in engineering contexts. For this new type of initial-boundary value problem \eqref{eq1.1} including Dirac measures, we provide the definition of solutions in the sense of distribution, which is one of the main difficulties and results in our work. We also establish the well-posedness of the solutions for our physical model \eqref{eq1.1}, and offer a more rigorous mathematical basis for engineering. In the end, we compare these two physical models from a mathematical theoretical perspective and conclude that the physical model \eqref{eq1.1} is superior to the one under the classical solution approach, which is of significant guidance for engineering practices. In the fourth section, we summarize our main conclusions and raise some further questions.
\section{The wave equation with one elastic support in a classical sense}\label{sec2}
Here, we introduce a new perspective of the specific problem in engineering, which includes an established classical solution theory corresponding to its physical model.
\subsection{Mathematical modeling of the wave equation with one elastic support in a classical sense}\label{subsec2.1}
By achieving an intuitive understanding in physical form of \eqref{eq1.1}, we directly divide a string into several parts according to the position of elastic supports. To begin with, we consider the wave equation with a single elastic support in the sense of classical solution for convenience and the elastic support is placed at $x=l (0<l<L)$. Clearly, the entire string is divided by the elastic support into two segments $x\in[0,l]$ and $x\in[l,L]$. For the first segment of the string, it satisfies the wave equation $u_{tt}-a^{2} u_{x x}=0$. Since the left end is fixed, it satisfies $u(0,t)=0$. Meanwhile, its right end is fixed on the elastic support, implying that the extension and contraction of the support obey Hooke's Law. If the original position of the elastic support is $u=0$, then the value of $u$ at the endpoint $x=l$ represents the extension of the support at that point.

According to Hooke's Law, the tension force of the string on the support is $\mathsf{K}u$, where $\mathsf{K}$ is the stiffness coefficient. This force is equal in magnitude but opposite in sign to the resultant force acting jointly from the left and right segments of the string. However, we notice that the vertical component of the tension force exerted by the support on the left segment of the string is $T\frac{\partial u}{\partial x}(l-0,t)$, and the vertical component of the tension force on the right segment of the string is $-T\frac{\partial u}{\partial x}(l+0,t)$. Therefore, at the position of elastic support with motion only in the vertical direction, the boundary condition is reduced to
\begin{equation}\label{eq2.3}
-u_{x}\left(l-0, t\right)+u_{x}\left(l+0, t\right)=\sigma u(l, t),
\end{equation}
where $\sigma=\frac{\mathsf{K}}{T}$ is a known positive constant. As a part of the entire string, this segment naturally satisfies the initial conditions $u(x, 0)=\varphi(x)$ and $u_{t}(x, 0)=\psi(x)$. Therefore, the initial boundary value problem satisfied by the string at the left end of the elastic support is as follows
\begin{equation}\label{eq2.4}
\begin{array}{l}
\left\{\begin{array}{l}
u_{tt}-a^{2} u_{x x}=0,  0<x<l,  t>0,\\
u(0, t)=0,  -u_{x}\left(l-0, t\right)+u_{x}\left(l+0, t\right)=\sigma u(l, t),\\
u(x, 0)=\varphi(x),  u_{t}(x, 0)=\psi(x).
\end{array}\right.
\end{array}
\end{equation}
Similarly, the initial boundary value problem satisfied by the string at the right end of the elastic support is
\begin{equation}\label{eq2.5}
\begin{array}{l}
\left\{\begin{array}{l}
u_{tt}-a^{2} u_{x x}=0,  l<x<L,  t>0,\\
u(L, t)=0,  -u_{x}\left(l-0, t\right)+u_{x}\left(l+0, t\right)=\sigma u(l, t),\\
u(x, 0)=\varphi(x),  u_{t}(x, 0)=\psi(x).
\end{array}\right.
\end{array}
\end{equation}
These two initial boundary value problems should be coupled together through condition \eqref{eq2.3}.
Moreover, these two segments should also satisfy the continuity condition
\begin{equation}\label{eq2.2}
u\left(l-0, t\right)=u\left(l+0, t\right)
\end{equation}
at $x=l$.
\par
Therefore, the problem of string vibration involving a concentrated external force is transformed into solving the coupled problems \eqref{eq2.3} \eqref{eq2.4} \eqref{eq2.5} and \eqref{eq2.2}.
It can be noticed, in modeling the vibration problem of a string with elastic support, the advantage of this model compared to \eqref{eq1.1} is that the equation no longer includes external force terms with a singular Dirac measure as the source term. The drawback, however, is that a single equation with Dirichlet initial-boundary value problem is transformed into two coupled initial-boundary value problems with third-type boundary conditions, resulting in a greater number of and more complex boundary conditions being provided. Nevertheless, this new model allows the original problem to be studied under established theories, thereby also providing new insights for the study of \eqref{eq1.1}. Of course, the solution to this coupled problem is not that immediately clear.

\subsection{Uniqueness and stability of the classical solution}\label{subsec2.1}
\par
Firstly, we are going to consider the uniqueness and stability of the solution to the boundary value problems if there exists a classical solution $u\in C^2$ to \eqref{eq2.3}-\eqref{eq2.2} for investigating whether there exist multiple solutions, and whether the solution depends continuously on the initial-boundary conditions.
\par
We assume that $u(x,t)$ is $C^1$ with respect to time $t$ at $x=l$ and introduce the energy integrals for equations \eqref{eq2.3} \eqref{eq2.4} \eqref{eq2.5} and \eqref{eq2.2}
\begin{equation}\label{eq2.57}
E(t)=\int_{0}^{l} (u_{t}^{2}+a^{2} u_{x}^{2}) \mathrm{d} x+\int_{l}^{L}(u_{t}^{2}+a^{2} u_{x}^{2}) \mathrm{d} x+a^{2} \sigma u^{2}(l, t).
\end{equation}
We observe the variation of $E(t)$ with respect to time $t$:
\begin{equation}\label{eq2.58}
\begin{aligned}
\frac{\mathrm{d} E(t)}{\mathrm{d} t}= & \int_{0}^{l} \left(2 u_{t}u_{t t}+2 a^{2} u_{x} u_{x t}\right) \mathrm{d} x+ \int_{l}^{L} \left(2 u_{t}u_{t t}+2 a^{2} u_{x} u_{x t}\right) \mathrm{d} x+2 a^{2} \sigma u(l, t) u_{t}(l, t) \\
= & 2 \int_{0}^{l} \left(u_{t} u_{t t}+a^{2}\left[\left(u_{x} u_{t}\right)_{x}-u_{t} u_{x x}\right]\right) \mathrm{d} x+2 \int_{l}^{L} \left(u_{t} u_{t t}+a^{2}\left[\left(u_{x} u_{t}\right)_{x}-u_{t} u_{x x}\right]\right) \mathrm{d}x \\
&+2 a^{2} \sigma u(l, t) u_{t}(l, t) \\
=&2a^{2}u_{x}(l-0,t)u_{t}(l-0,t)-2a^{2}u_{x}(l+0,t)u_{t}(l+0,t)+2 a^{2} \sigma u(l, t) u_{t}(l, t) \\
=&2a^{2}u_{x}(l+0,t)u_{t}(l-0,t)-2a^{2}u_{x}(l+0,t)u_{t}(l+0,t)-2 a^{2} \sigma u(l, t) u_{t}(l-0, t)\\
&+2 a^{2} \sigma u(l, t) u_{t}(l, t)\\
=&0,
\end{aligned}
\end{equation}
where we use \eqref{eq2.3} in the second-to-last equal sign. This implies $\frac{\mathrm{d}E(t)}{\mathrm{d}t}=0$, hence $E(t)$ is a constant with respect to $t$, i.e. $E(t) \equiv E(0)$. The fact of this conservation of total energy leads to
\begin{theorem}\label{thm2.1}
If a classical solution to the initial-boundary value problem \eqref{eq2.3} \eqref{eq2.4} \eqref{eq2.5} and\\\eqref{eq2.2} exists, it must be unique.
\end{theorem}
\begin{proof}
Let $u_1$ and $u_2$ be two solutions to the given initial-boundary value problem \eqref{eq2.3}-\eqref{eq2.2}.\\Their difference $u=u_1-u_2$ satisfies the corresponding homogeneous equation and homogeneous boundary conditions as follows
\begin{equation}\label{eq2.59}
\begin{array}{l}
\left\{\begin{array}{l}
u_{tt}-a^{2} u_{xx}=0,  0<x<l, l<x<L,  t>0,\\
u(0,t)=0,  u(L,t)=0,\\
u(x, 0)=0,  u_{t}(x, 0)=0,\\
\end{array}\right. \\
\end{array}\\
\end{equation}
with
\begin{center}
$u\left(l-0, t\right)=u\left(l+0, t\right)$ and $-u_{x}\left(l-0, t\right)+u_{x}\left(l+0, t\right)=\sigma u(l, t)$.
\end{center}
Thus, at the initial time $t=0$, we have
\begin{equation}\label{eq2.60}
E(0)=\int_{0}^{l}\left( u_{t}^{2}(x, 0)+a^{2} u_{x}^{2}(x, 0)\right) \mathrm{d} x+\int_{l}^{L}\left(u_{t}^{2}(x, 0)+a^{2} u_{x}^{2}(x, 0)\right) \mathrm{d} x+a^{2} \sigma u^{2}(l, 0),
\end{equation}
and according to $\eqref{eq2.59}_3$, we notice that
\begin{equation}\label{eq2.61}
u_{x}(x, 0)=\lim _{\Delta x \rightarrow 0} \frac{u(x+\Delta x, 0)-u(x, 0)}{\Delta x}=0.
\end{equation}
Combining with the initial and boundary conditions of the problem \eqref{eq2.59}, we obtain $E(0)=0$. Therefore, we have
\begin{equation}\label{eq2.62}
E(t)=\int_{0}^{l} \left(u_{t}^{2}+a^{2} u_{x}^{2}\right) \mathrm{d} x+\int_{l}^{L}\left(u_{t}^{2}+a^{2} u_{x}^{2}\right) \mathrm{d} x+a^{2} \sigma u^{2}(l, t)=0.
\end{equation}
That is, $u_t=u_x=0$. Furthermore, since $u=0$ when $t=0$, we conclude that $u(x, t) \equiv 0$. Thus, the uniqueness of solution to equations \eqref{eq2.3}-\eqref{eq2.2} is established.
\end{proof}
\par
Through the energy method, we can further demonstrate the stability of the initial conditions for the initial-boundary value problem \eqref{eq2.3}-\eqref{eq2.2}.
\begin{theorem}\label{thm2.2}
The solution $u(x,t)$ to the initial-boundary value problem \eqref{eq2.3} \eqref{eq2.4} \eqref{eq2.5} and\\ \eqref{eq2.2} is stable with respect to the initial value $(\varphi,\psi)$ in the following sense: for any given $\varepsilon>0$, there exists $\eta>0$ depending only on $\varepsilon$ and $T$, such that if
\begin{equation}\label{eq2.63}
\begin{array}{l}
\left\|\psi_{1}-\psi_{2}\right\|_{L^{2}((0, l))}+\left\|\psi_{1}-\psi_{2}\right\|_{L^{2}((l,L))}\leq \eta, \\
\left\|\varphi_{1}-\varphi_{2}\right\|_{L^{2}((0, l))}+\left\|\varphi_{1}-\varphi_{2}\right\|_{L^{2}((l, L))}\leq \eta, \\
\left\|\varphi_{1_x}-\varphi_{2_{x}}\right\|_{L^{2}((0, l))}+\left\|\varphi_{1_x}-\varphi_{2_{x}}\right\|_{L^{2}((l, L))} \leq \eta,
\end{array}
\end{equation}
then for $0\leq t \leq T$, the difference between the solutions $u_1$ with initial values $(\varphi_1,\psi_1)$ and $u_2$ with initial values $(\varphi_2,\psi_2)$ satisfies
\begin{equation}\label{eq2.64}
\begin{array}{l}
\left\|u_{1}-u_{2}\right\|_{L^{2}((0, l))}+\left\|u_{1}-u_{2}\right\|_{L^{2}((l, L))}\leq \varepsilon, \\
\left\|u_{1x}-u_{2x}\right\|_{L^{2}((0, l))}+\left\|u_{1x}-u_{2x}\right\|_{L^{2}((l, L))}\leq \varepsilon, \\
\left\|u_{1t}-u_{2t}\right\|_{L^{2}((0, l))}+\left\|u_{1t}-u_{2t}\right\|_{L^{2}((l, L))} \leq \varepsilon.
\end{array}
\end{equation}
\end{theorem}
\begin{proof}
Let $w=u_{1}(x,t)-u_{2}(x,t)$. Then $w(x,t)$ satisfies
\begin{equation}\label{eq2.65}
\left\{\begin{array}{l}
w_{tt}=a^{2}w_{x x}, \quad 0<x<l, l<x<L,  t>0,\\
w(0, t)=0,  w_{x}(L,t)=0,\\
w(x, 0)=\tilde{\varphi },  w_{t}(x, 0)=\tilde{\psi},\\
\end{array}\right.
\end{equation}
where $\tilde{\varphi}=\varphi_{1}-\varphi_{2}, \tilde{\psi}=\psi_{1}-\psi_{2}$, and it also satisfies the conditions
\begin{center}
$w\left(l-0, t\right)=w\left(l+0, t\right)$ and $-w_{x}\left(l-0, t\right)+w_{x}\left(l+0, t\right)=\sigma w(l, t)$.
\end{center}
From \eqref{eq2.57} and $\eqref{eq2.58}$, we obtain
\begin{equation}\label{eq2.66}
E(t)=\int_{0}^{l} w_{t}^{2}(x,t)+a^{2} w_{x}^{2}(x,t) \mathrm{d} x+\int_{l}^{L} w_{t}^{2}(x,t)+a^{2} w_{x}^{2}(x,t) \mathrm{d} x+a^2\sigma w^{2}(l,t) \text{and} \frac{\mathrm{d}E(t)}{\mathrm{d}t}=0.
\end{equation}
Thus, $E(t)$ is a constant with respect to $t$, i.e. $E(t) \equiv E(0)$.
\par
Further, we can obtain an estimate for $L^2$ norm of the function $w(x,t)$. Let us denote
\begin{equation}\label{eq2.67}
\begin{aligned}
E_{0}(t)=&\int_{0}^{l} w^{2}(x, t) \mathrm{d} x+\int_{l}^{L} w^{2}(x, t) \mathrm{d} x \\
=&{\|w(\cdot,t)\|_{L^{2}((0, l))}^{2}}+{\|w(\cdot,t)\|_{L^{2}((l, L))}^{2}}.
\end{aligned}
\end{equation}
Taking the derivative with respect to $t$, we have
\begin{equation}\label{eq2.68}
\begin{aligned}
\frac{\mathrm{d} E_{0}(t)}{\mathrm{d} t}&=2 \int_{0}^{l} w w_{t} \mathrm{d} x+2 \int_{l}^{L} w w_{t} \mathrm{d} x \\
&\leq \int_{0}^{l} w^{2} \mathrm{d} x+\int_{0}^{l} w_{t}^{2} \mathrm{d} x +\int_{l}^{L} w^{2} \mathrm{d} x+\int_{l}^{L} w_{t}^{2} \mathrm{d} x\\
&\leq E_{0}(t)+E(t).
\end{aligned}
\end{equation}
Then, multiplying both sides of the equation by $\mathrm{e}^{-t}$ and using Gronwall's inequality, we have
\begin{equation}\label{eq2.69}
\frac{\mathrm{d}}{\mathrm{d} t}\left(\mathrm{e}^{-t} E_{0}(t)\right) \leqslant \mathrm{e}^{-t} E(t).
\end{equation}
Integrating both sides of the equation from $0$ to $t$, we also obtain
\begin{equation}\label{eq2.70}
\begin{aligned}
E_{0}(t) & \leqslant \mathrm{e}^{t} E_{0}(0) + \int_{0}^{t} \mathrm{e}^{t-\tau}E(\tau) \mathrm{d} \tau \\
& = \mathrm{e}^{t}E_{0}(0) + (\mathrm{e}^{t}-1)E(0),
\end{aligned}
\end{equation}
where
\begin{center}
$E(0)=\|\tilde\psi\|_{L^{2}((0, l))}^{2}+\|\tilde\psi\|_{L^{2}((l, L))}^{2}+a^{2}\left\|\tilde\varphi_{x}\right\|_{L^{2}((0, l))}^{2}+a^{2}\left\|\tilde\varphi_{x}\right\|_{L^{2}((l, L))}^{2}+a^{2} \sigma \tilde\varphi^{2}(l)$.
\end{center}
Then we have
\begin{equation}\label{eq2.71}
{\|w(\cdot,t)\|_{L^{2}((0, l))}^{2}}+{\|w(\cdot,t)\|_{L^{2}((l,L))}^{2}}=E_{0}(t) \leq \mathrm{e}^{t}E_{0}(0)+(\mathrm{e}^{t}-1)E(0),
\end{equation}
where
\begin{center}
$E_{0}(0)=\|\tilde\varphi(x)\|_{L{((0, l))}^{2}}^{2}+\|\tilde\varphi(x)\|_{L{((l, L))}^{2}}^{2}$.
\end{center}
So when \eqref{eq2.63} holds, taking $\eta=\varepsilon /\sqrt{\mathrm e^t(2+a^2+a^2\sigma L^{-1})-(1+a^2+a^2\sigma L^{-1})}$, we have
\begin{equation}\label{eq2.72}
\|w\|_{L^{2}((0,l))}^{2}+\|w\|_{L^{2}((l,L))}^{2}\leq\varepsilon^2.
\end{equation}
This yields the desired conclusion.
\end{proof}

\subsection{Existence of classical solution}\label{subsec2.2}
Afterward, by further observing the boundary condition \eqref{eq2.3} derived from the force analysis, we can arrive at the following conclusions as one of our main results Theorem \ref{thm2.3}.
\begin{theorem}\label{thm2.3}
Given that $u(l-0,t)=u(l+0,t)$, we may denote them as $h(t)$ being an undetermined function dependent on $t$. For the initial-boundary value problem \eqref{eq2.3} \eqref{eq2.4} \eqref{eq2.5} and\\\eqref{eq2.2}, we have\\
(1) If $u\in C^1([0,L]\times(0,T))$, then $-u_x(l-0,t)+u_x(l+0,t)=0=\sigma h(t)$ and $h(t)\equiv0$;\\
(2) The necessary condition for the coupled problem \eqref{eq2.3}-\eqref{eq2.2} to have $u\in C^1([0,L]\times(0,T))$ is that $\varphi(x)\in C^2([0,L])$ and $\varphi(l)=0$. And thus the equations \eqref{eq2.3}-\eqref{eq2.2} does not admit a global $C^2$ solution $u$ on $[0,L]$ for general $\varphi(x)\in C^2([0,L])$.
\end{theorem}
\begin{proof}
\par
(1) and (2) can be concluded successively given that $\sigma$ is a positive constant.
\par
(3) Supposing that there exists a global $C^2$ solution $u$ to \eqref{eq2.3}-\eqref{eq2.2} on $[0,L]$, then we should have $h(t)\equiv 0$ holding true due to \eqref{eq2.3}. Then, the problem \eqref{eq2.3}-\eqref{eq2.2} turns into the most typical string vibration equation problem with two fixed ends and no external force terms. It is well known that there is a global $C^2$ solution $u$ for this problem with $\varphi(x)\in C^3, \psi(x)\in C^2$, $\varphi(0)=\varphi(l)=\varphi''(0)=\varphi''(l)=0$, and $\psi(0)=\psi(l)=0$. However, we are not able to find a fixed point within the string that remains constantly zero, such as the initial-boundary value problem
\begin{equation*}
\begin{array}{l}
\left\{\begin{array}{l}
u_{tt}-a^{2} u_{x x}=0,  0<x<2\pi,  t>0,\\
u(0, t)=0,  u(2\pi, t)=0,\\
u(x, 0)=\sin x,  u_{x}(x, 0)=0.\\
\end{array}\right. \\
\end{array}\\
\end{equation*}
We find that there is no fixed point when $l \ne\pi$, which leads to $h(t)\not\equiv0$. For the sake of contradiction, we have completed the proof.
\end{proof}
Through this theorem, we have mathematically proven a major conclusion: there is an essential difference depending on whether internal elastic supports are placed within the string. As the previous theoretical findings are insufficient, taking further exploration of the problem \eqref{eq2.3}-\eqref{eq2.2} is necessary. Additionally, due to the absence of a global $C^2$ solution for the problem \eqref{eq2.3}-\eqref{eq2.2}, this problem at most has a solution that is piecewise $C^2$ and global $C^1$.
\par
We note that \eqref{eq2.4} and \eqref{eq2.5} have the same form. We might as well start from solving \eqref{eq2.5}. For the initial-boundary value problem \eqref{eq2.5}, this problem admits nonhomogeneous boundary conditions. Meanwhile, we can turn the boundary conditions into a homogeneous case through an appropriate transform of the unknown function.
Let
\begin{equation}\label{eq2.19}
U(x, t)=\frac{L-x}{L-l} h(t).
\end{equation}
which is a function satisfying the boundary conditions. Then, we introduce a new unknown function
\begin{equation}\label{eq2.20}
V(x, t)=u(x, t)-U(x, t)=u(x, t)-\frac{L-x}{L-l} h(t),
\end{equation}
which solves the equation
\begin{equation}\label{eq2.21}
V_{t t}-a^{2} V_{x x}=-\frac{L-x}{L-l} h^{\prime \prime}(t)
\end{equation}
with the following nonhomogeneous initial conditions,
\begin{equation}\label{eq2.22}
V(x,0) =\varphi(x)-U(x, 0)=\varphi(x)-\frac{L-x}{L-l} h(0),
\end{equation}
\begin{equation}\label{eq2.23}
V_{t}(x,0) =\psi(x)-U_{t}(x, 0) =\psi(x)-\frac{L-x}{L-l} h^{\prime}(0),
\end{equation}
and homogeneous boundary conditions. Thus the initial-boundary value problem \eqref{eq2.5} turns into
\begin{equation}\label{eq2.24}
\left\{\begin{array}{l}
V_{t t}-a^{2} V_{x x}=g(x,t),  l<x<L,  t>0,\\
V(L, t)=0,  V(l, t)=0,\\
V(x, 0)=\varphi_2(x),\\
V_{t}(x, 0)=\psi_2(x),
\end{array}\right.
\end{equation}
where $$g(x,t)=-\frac{L-x}{L-l} h^{\prime \prime}(t),$$ $$\varphi_2(x)=\varphi(x)-\frac{L-x}{L-l} h(0)$$ and $$\psi_2(x)=\psi(x)-\frac{L-x}{L-l} h^{\prime}(0).$$ Suppose that
\begin{equation}\label{eq2.42}
\varphi_{2}(x)\in C^3, \psi_{2}(x)\in C^2
\end{equation}
and
\begin{equation}\label{eq2.43}
\varphi_{2}(L)=\varphi_{2}(l)=\varphi_{2}''(L)=\varphi_{2}''(l)=\psi_{2}(L)=\psi_{2}(l)=0.
\end{equation}
With the principle of superposition, the above initial boundary value problem \eqref{eq2.24} can be decomposed into the following coupled initial-boundary value problems
\begin{equation}\label{eq2.25}
\begin{array}{l}
\left\{\begin{array}{l}
V_{3tt}-a^{2} V_{3x x}=0,  l<x<L,  t>0,\\
V_{3}(L, t)=0,  V_{3}(l, t)=0,\\
V_{3}(x, 0)=\varphi_{2}(x),  V_{3t}(x, 0)=\psi_{2}(x),\\
\end{array}\right. \\
\end{array}\\
\end{equation}
together with
\begin{equation}\label{eq2.26}
\begin{array}{l}
\left\{\begin{array}{l}
V_{4tt}-a^{2} V_{4x x}=g(x,t),  l<x<L,  t>0,\\
V_{4}(L, t)=0,  V_{4}(l, t)=0,\\
V_{4}(x, 0)=0,  V_{4t}(x, 0)=0,\\
\end{array}\right. \\
\end{array}\\
\end{equation}
and clearly $$V=V_3+V_4.$$
\par
First of all, for the boundary value problem \eqref{eq2.25}, we make a variable substitution by letting $y=x-l$, then we have
\begin{equation}\label{eq2.27}
\begin{array}{l}
\left\{\begin{array}{l}
V_{3tt}-a^{2} V_{3yy}=0,  0<y<L-l,  t>0,\\
V_{3}(L-l, t)=0,  V_{3}(0, t)=0,\\
V_{3}(y, 0)=\varphi_2(y+l),  V_{3t}(y, 0)=\psi_2(y+l).\\
\end{array}\right. \\
\end{array}\\
\end{equation}
Thus, solving the initial-boundary value problem \eqref{eq2.25} turns into solving the initial-boundary value problem \eqref{eq2.27}. Now, we seek a non-trivial particular solution of equation \eqref{eq2.27} that allows for separation of variables
\begin{equation}\label{eq2.28}
V_3(y,t)=Y(y)T(t).
\end{equation}
Additionally, it must satisfy homogeneous boundary conditions. Here, $Y(y)$ and $T(t)$ respectively represent undetermined functions depending only on $y$ and $t$.
\par
Substituting \eqref{eq2.28} into $V_{3tt} = a^2V_{3yy}$, we obtain
\begin{equation}\label{eq2.29}
\frac{T''(t)}{a^2T(t)}=\frac{Y''(y)}{Y(y)}=-\lambda.
\end{equation}
In equation \eqref{eq2.27}, we obtain
\begin{equation}\label{eq2.30}
T''(t)+\lambda a^2T(t)=0,
\end{equation}
and
\begin{equation}\label{eq2.31}
Y''(y)+\lambda Y(y)=0.
\end{equation}
\par
Thus, the equation $V_{3tt} = a^2V_{3yy}$ is separated into two ordinary differential equations, with one involving only the independent variable $t$ and the other involving only the independent variable $y$. By solving these two equations, we can determine $T(t)$ and $Y(y)$, thereby obtaining a particular solution \eqref{eq2.28} for the equation \eqref{eq2.27}.
\par
To ensure that this solution is a nontrivial one that meets the homogeneous boundary conditions, we must find a nontrivial solution that satisfies the boundary conditions for equation \eqref{eq2.31}:
\begin{equation}\label{eq2.32}
Y(0)=0,   Y(L-l)=0.
\end{equation}
The general solution of equation \eqref{eq2.31} varies from $\lambda>0$, $\lambda=0$ to $\lambda<0$. We can easily know that when $\lambda=0$ or $\lambda<0$, $Y(y)$ must be constantly zero which means a nontrivial solution cannot be obtained.
\par
However, when $\lambda>0$, the general solution of equation \eqref{eq2.31} can be written as
\begin{equation}\label{eq2.33}
Y(y)=C_1\mathrm{cos}\sqrt{\lambda}y+C_2\mathrm{sin}\sqrt{\lambda}y.
\end{equation}
From the boundary condition $Y(0)=0$, we know that $C_1=0$. Then, from the second boundary condition of \eqref{eq2.32}, we obtain
\begin{equation}\label{eq2.34}
C_{2}\sin \sqrt{\lambda}(L-l)=0.
\end{equation}
To make $Y(y)$ a nontrivial solution, $\lambda$ should satisfy
\begin{equation}\label{eq2.35}
\sin \sqrt{\lambda}(L-l)=0.
\end{equation}
Therefore, for the problem \eqref{eq2.31} and \eqref{eq2.32}, we have
\begin{equation}\label{eq2.36}
\lambda_k=\lambda=\frac{k^2\pi^2}{(L-l)^2}  (k=1,2,...)
\end{equation}
and
\begin{equation}\label{eq2.37}
Y_k=C_k\mathrm{sin}\sqrt{\lambda_k}y=C_{k} \sin \frac{k \pi}{L-l} y  (k=1,2,...).
\end{equation}
Substituting $\lambda=\lambda_k$ into equation \eqref{eq2.30}, we obtain
\begin{equation}\label{eq2.38}
T_k(t)=A_k\mathrm{cos}\frac{k\pi a}{L-l}t+B_k\mathrm{sin}\frac{k\pi a}{L-l}t  (k=1,2,...),
\end{equation}
where $A_k$ and $B_k$ are arbitrary constants.
\par
Thus, we have the following separated variable solutions satisfying the homogeneous boundary conditions for equation \eqref{eq2.27}. We take an appropriate linear combination of these particular solutions to obtain the solution to the initial-boundary value problem
\begin{equation}\label{eq2.39}
V_{3}(y, t)=\sum_{k=1}^{\infty}\left(A_{k} \cos \frac{k \pi a}{L-l} t+B_{k} \sin \frac{k \pi a}{L-l} t\right) \sin \frac{k \pi}{L-l} y.
\end{equation}
In other words, we still need to determine the constants $A_k$ and $B_k$.
\par
Next, let us specify the constants $A_k$. According to the initial condition, in order to take the initial value $\varphi_{2}(y+l)$ for $V_3(y,t)$ when $t=0$, the following condition should hold
\begin{equation}\label{eq2.40}
\varphi_{2}(y+l)=\sum_{k=1}^{\infty} A_{k} \sin \frac{k \pi}{L-l} y.
\end{equation}
We multiply both sides of equation \eqref{eq2.40} by $\mathrm{sin}\frac{k\pi}{L-l}y$, then integrate the result from $0$ to $L-l$. This yields
\begin{equation}\label{eq2.41}
A_{k}=\frac{2}{L-l} \int_{0}^{L-l} \varphi_{2}(\xi+l) \sin \frac{k \pi}{L-l} \xi \mathrm{d} \xi.
\end{equation}
\par
To determine the coefficients $B_k$, we can differentiate term by term in equation \eqref{eq2.39},
\begin{equation}\label{eq2.44}
\begin{aligned}
V_{3t}(y, t)=\sum_{k=1}^{\infty} & \left(A_{k}(-\sin \frac{k \pi a}{L-l} t)\frac{k \pi a}{L-l}\right.\left.+B_{k}(\cos \frac{k \pi a}{L-l} t )\frac{k \pi a}{L-l}\right) \cdot \sin \frac{k \pi}{L-l} y.
\end{aligned}
\end{equation}
Thus, according to the initial condition $\eqref{eq2.25}_3$, it should hold that
\begin{equation}\label{eq2.45}
\psi_{2}(y+l)=\sum_{k=1}^{\infty} B_{k}\frac{k \pi a}{L-l} \sin \frac{k \pi}{L-l} y.
\end{equation}
Similarly, we can obtain
\begin{equation}\label{eq2.46}
B_{k}=\frac{2}{k \pi a} \int_{0}^{L-l} \psi_{2}(\xi+l) \sin \frac{k \pi}{L-l} \xi \mathrm{d} \xi.
\end{equation}
Substituting $A_k$ and $B_k$ into equation \eqref{eq2.39}, we obtain the formal solution to the initial-boundary value problem \eqref{eq2.27}, leading to the formal solution to \eqref{eq2.25}
\begin{equation}\label{eq2.47}
\begin{aligned}
V_{3}(x, t)
&=\sum_{k=1}^{\infty}(\frac{2}{L-l} \int_{l}^{L}(\varphi(\eta)-\frac{L-\eta}{L-l} h(0))\sin \frac{k \pi}{L-l}(\eta-l) \mathrm{d}\eta\cdot
\cos \frac{k \pi a}{L-l} t\\
&+\frac{2}{k \pi a} \int_{l}^{L}(\psi(\eta)-\frac{L-\eta}{L-l} h^{\prime}(0))\sin \frac{k \pi}{L-l}(\eta-l) \mathrm{d} \eta\cdot \sin \frac{k \pi a}{L-l} t)\cdot \sin \frac{k \pi}{L-l}(x-l).\\
\end{aligned}
\end{equation}
\par
Next, for the initial-boundary value problem \eqref{eq2.26}, by substituting $y=x-l$, this boundary value problem turns into
\begin{equation}\label{eq2.48}
\begin{array}{l}
\left\{\begin{array}{l}
V_{4tt}-a^{2} V_{4yy}=g(y+l,t),  0<y<L-l,  t>0,\\
V_{4}(L-l, t)=0,  V_{4}(0, t)=0,\\
V_{4}(y, 0)=0,  V_{4t}(y, 0)=0.\\
\end{array}\right. \\
\end{array}\\
\end{equation}
According to the principle of homogenization, if $W(y,t;\tau)$ is a solution to the initial-boundary value problem of the homogeneous equation
\begin{equation}\label{eq2.49}
\begin{array}{l}
\left\{\begin{array}{l}
W_{tt}-a^{2} W_{yy}=0,  t>\tau,\\
W(L-l, t)=0,  W(0, t)=0,\\
W(y, \tau)=0,  W_{t}(y, \tau)=g(y+l,\tau),\\
\end{array}\right. \\
\end{array}\\
\end{equation}
then
\begin{equation}\label{eq2.50}
V_{4}(y,t)=\int_{0}^{t} w(y, t ; \tau) \mathrm d \tau
\end{equation}
is a solution to the initial-boundary value problem \eqref{eq2.48}. Let $t^{\prime}=t-\tau$. Then \eqref{eq2.49} becomes
\begin{equation}\label{eq2.51}
\begin{array}{l}
\left\{\begin{array}{l}
W_{tt}-a^{2} W_{yy}=0,  t^{\prime}>0,\\
W(L-l, t^{\prime})=0,  W(0, t^{\prime})=0,\\
W(y, 0)=0,  W_{t^{\prime}}(y, 0)=g(y+l,\tau).\\
\end{array}\right. \\
\end{array}\\
\end{equation}
Because both the equation and the boundary conditions are homogeneous, we have
\begin{equation}\label{eq2.52}
\begin{aligned}
W\left(y, t ; \tau\right)
&= \sum_{k=1}^{\infty} B_{k}(\tau) \sin \frac{k \pi a}{L-l} t^{\prime} \sin \frac{k \pi}{L-l} y \\
&=\sum_{k=1}^{\infty} \left(\frac{2}{k \pi a} \int_{0}^{L-l} g(\xi+l, \tau) \sin \frac{k \pi}{L-l} \xi\mathrm{d} \xi\right) \cdot \sin \frac{k \pi a}{L-l}(t-\tau)\sin \frac{k \pi}{L-l} y.
\end{aligned}
\end{equation}
So, we have
\begin{equation}\label{eq2.53}
\begin{aligned}
V_{4}(x, t)
=&\sum_{k=1}^{\infty} (\int_{0}^{t}\frac{2}{k \pi a} \int_{l}^{L}(-\frac{L-\eta}{L-l} h^{\prime \prime}(\tau))\sin \frac{k \pi}{L-l}(\eta-l) \mathrm{d} \eta \\
&\sin \frac{k \pi a}{L-l}(t-\tau) \mathrm{d} \tau)\cdot \sin \frac{k \pi}{L-l}(x-l).
\end{aligned}
\end{equation}
Due to $V=V_3+V_4$ and \eqref{eq2.20}, it can be obtained that
\begin{equation}\label{eq2.54}
\begin{aligned}
u(x, t)=&\sum_{k=1}^{\infty}(\frac{2}{L-l} \int_{l}^{L}(\varphi(\eta)-\frac{L-\eta}{L-l} h(0)) \sin \frac{k \pi}{L-l}(\eta-l)\mathrm{d} \eta\cdot\cos \frac{k \pi a}{L-l} t\\
&                   +\frac{2}{k \pi a} \int_{l}^{L}(\psi(\eta)-\frac{L-\eta}{L-l} h^{\prime}(0))\sin \frac{k \pi}{L-l}(\eta-l)\mathrm{d}\eta \cdot \sin \frac{k \pi a}{L-l} t)\\
&                    \cdot \sin \frac{k \pi}{L-l}(x-l)\\
&+\sum_{k=1}^{\infty}(\int_{0}^{t} \frac{2}{k \pi a} \int_{l}^{L}(-\frac{L-\eta}{L-l} h^{\prime \prime}(\tau)) \sin \frac{k \pi}{L-l} (\eta-l) \mathrm{d} \eta \cdot \sin \frac{k \pi a}{L-l}(t-\tau) \mathrm{d} \tau)\\
&                    \cdot\sin \frac{k \pi}{L-l}(x-l)+\frac{L-x}{L-l} h(t).
\end{aligned}
\end{equation}
\par
When it comes to the initial-boundary problem \eqref{eq2.4}, by simply repeating the above steps, \eqref{eq2.4} turns into
\begin{equation}\label{eq2.11}
\begin{array}{l}
\left\{\begin{array}{l}
V_{tt}-a^{2} V_{x x}=f(x,t),  0<x<l,  t>0,\\
V(0, t)=0,  V(l, t)=0,\\
V(x, 0)=\varphi_{1}(x),  V_{t}(x, 0)=\psi_{1}(x),\\
\end{array}\right. \\
\end{array}\\
\end{equation}
in which
\begin{center}
$f(x,t)=-\frac{x}{l}h''(t)$,
\end{center}
\begin{center}
$\varphi_{1}(x)=\varphi(x)-\frac{x}{l}h(0)$
\end{center}
and
\begin{center}
$\psi_{1}(x)=\psi(x)-\frac{x}{l}h'(0)$.
\end{center}
We suppose that
\begin{equation}\label{eq2.12}
\varphi_{1}(x)\in C^3, \psi_{1}(x)\in C^2
\end{equation}
and
\begin{equation}\label{eq2.13}
\varphi_{1}(0)=\varphi_{1}(l)=\varphi_{1}''(0)=\varphi_{1}''(l)=\psi_{1}(0)=\psi_{1}(l)=0.
\end{equation}
Ultimately, it can be obtained that
\begin{equation}\label{eq2.18}
\begin{aligned}
u(x, t)&=\sum_{k=1}^{\infty}(\frac{2}{l} \int_{0}^{l}(\varphi(\xi)-\frac{\xi}{l} h(0)) \sin \frac{k \pi}{l} \xi \mathrm{d} \xi \cdot\cos \frac{k \pi a}{l} t\\
&                               +\frac{2}{k \pi a} \int_{0}^{l}(\psi(\xi)-\frac{\xi}{l} h^{\prime}(0))\sin \frac{k \pi}{l} \xi \mathrm{d}\xi\cdot \sin \frac{k \pi a}{l} t)\cdot \sin \frac{k \pi}{l}x\\
&+\sum_{k=1}^{\infty}(\int_{0}^{t} \frac{2}{k \pi a} \int_{0}^{l}(-\frac{\xi}{l} h^{\prime \prime}(\tau)) \sin \frac{k \pi}{l} \xi \mathrm{d} \xi \sin \frac{k \pi a}{l}(t-\tau) \mathrm{d} \tau)\cdot\sin \frac{k \pi}{l}x\\
&                               +\frac{x}{l} h(t).
\end{aligned}
\end{equation}
Meanwhile, the function $u$ given by \eqref{eq2.18} and \eqref{eq2.54} also need to satisfy the boundary condition \eqref{eq2.3}, so we have
\begin{equation}\label{eq2.55}
\begin{array}{l}
u(x,t)=\left\{\begin{array}{ll}
&\sum\limits_{k=1}^{\infty}(\frac{2}{l} \int_{0}^{l}(\varphi(\xi)-\frac{\xi}{l} h(0)) \sin \frac{k \pi}{l} \xi \mathrm{d} \xi \cdot\cos \frac{k \pi a}{l} t\\
&                 +\frac{2}{k \pi a} \int_{0}^{l}(\psi(\xi)-\frac{\xi}{l} h^{\prime}(0))\sin \frac{k \pi}{l} \xi \mathrm{d}\xi\cdot \sin \frac{k \pi a}{l} t)\cdot \sin \frac{k \pi}{l}x\\
&+\sum\limits_{k=1}^{\infty}(\int_{0}^{t} \frac{2}{k \pi a} \int_{0}^{l}(-\frac{\xi}{l} h^{\prime \prime}(\tau)) \sin \frac{k \pi}{l} \xi \mathrm{d} \xi \sin \frac{k \pi a}{l}(t-\tau) \mathrm{d} \tau)\cdot\sin \frac{k \pi}{l}x\\
&                 +\frac{x}{l} h(t), \text{if} 0<x<l,\\
&\sum\limits_{k=1}^{\infty}(\frac{2}{L-l} \int_{l}^{L}(\varphi(\eta)-\frac{L-\eta}{L-l} h(0)) \sin \frac{k \pi}{L-l}(\eta-l)\mathrm{d} \eta\cdot \cos \frac{k \pi a}{L-l} t\\
&                 +\frac{2}{k \pi a} \int_{l}^{L}(\psi(\eta)-\frac{L-\eta}{L-l} h^{\prime}(0))\sin \frac{k \pi}{L-l}(\eta-l)\mathrm{d}\eta \cdot \sin \frac{k \pi a}{L-l} t)\cdot \sin \frac{k \pi}{L-l}(x-l)\\
&+\sum\limits_{k=1}^{\infty}(\int_{0}^{t} \frac{2}{k \pi a} \int_{l}^{L}(-\frac{L-\eta}{L-l} h^{\prime \prime}(\tau)) \sin \frac{k \pi}{L-l} (\eta-l) \mathrm{d} \eta \sin \frac{k \pi a}{L-l}(t-\tau) \mathrm{d} \tau)\cdot\sin \frac{k \pi}{L-l}(x-l)\\
&                 +\frac{L-x}{L-l} h(t), \text{if} l<x<L,\\
\end{array}\right. \\
\end{array}\\
\end{equation}
in which $h(t)$ is determined by
\begin{equation}\label{eq2.56}
\begin{aligned}
&-\sum_{k=1}^{\infty}(\frac{2}{l} \int_{0}^{l}(\varphi(\xi)-\frac{\xi}{l} h(0))\sin \frac{k \pi}{l} \xi \mathrm{d} \xi \cdot\cos \frac{k \pi a}{l} t
+\frac{2}{k \pi a} \int_{0}^{l}(\psi(\xi)\\
&                     -\frac{\xi}{l} h^{\prime}(0))\sin \frac{k \pi}{l} \xi \mathrm{d}\xi\cdot\sin \frac{k \pi a}{l} t)\cdot(-1)^{k}\cdot \frac{k \pi}{l}\\
&-\sum_{k=1}^{\infty}(\int_{0}^{t} \frac{2}{k \pi a} \int_{0}^{l}(-\frac{\xi}{l} h^{\prime \prime}(\tau)) \sin \frac{k \pi}{l} \xi \mathrm{d} \xi \sin \frac{k \pi a}{l}(t-\tau) \mathrm{d} \tau)\cdot(-1)^{k}\cdot\frac{k \pi}{l}-\frac{1}{l} h(t)\\
&+\sum_{k=1}^{\infty}(\frac{2}{L-l} \int_{l}^{L}(\varphi(\eta)-\frac{L-\eta}{L-l} h(0)) \sin \frac{k \pi}{L-l}(\eta-l)\mathrm{d} \eta \cdot \cos \frac{k \pi a}{L-l}t\\
&                     +\frac{2}{k \pi a} \int_{l}^{L}(\psi(\eta)-\frac{L-\eta}{L-l} h^{\prime}(0))\sin \frac{k \pi}{L-l}(\eta-l)\mathrm{d}\eta \cdot \sin \frac{k \pi a}{L-l} t)\cdot\frac{k \pi}{L-l}\\
&+\sum_{k=1}^{\infty}(\int_{0}^{t} \frac{2}{k \pi a} \int_{l}^{L}(-\frac{L-\eta}{L-l} h^{\prime \prime}(\tau)) \sin \frac{k \pi}{L-l} (\eta-l) \mathrm{d} \eta \sin \frac{k \pi a}{L-l}(t-\tau) \mathrm{d} \tau)\\
&                      \cdot\frac{k \pi}{L-l}-\frac{1}{L-l} h(t)
=\sigma h(t).
\end{aligned}
\end{equation}
So far, we have obtained the formal solution of equations \eqref{eq2.3}-\eqref{eq2.2}.
\begin{theorem}\label{thm2.4}
If $\varphi(x)\in C^3,\psi(x)\in C^2$ and $\varphi(0)=\varphi(l)=\varphi(L)=0, \varphi''(0)=\varphi''(l)=\varphi''(L)=0, \psi(0)=\psi(l)=\psi(L)=0$, then the problem \eqref{eq2.3}-\eqref{eq2.2} has a formal analytical solution \eqref{eq2.55}, where $h(t)$ satisfies \eqref{eq2.56}.
\end{theorem}
\par
In order to answer whether the formal solution $u$ is piecewise $C^2$, we need to investigate the properties of the series \eqref{eq2.56}. Notice that when $t=0$, \eqref{eq2.56} reduces to an integral equation
\begin{equation}\label{eq111}
\begin{aligned}
&-\sum_{k=1}^{\infty}\left(\frac{2}{l} \int_{0}^{l}(\varphi(\xi)-\frac{\xi}{l}h(0))\sin \frac{k \pi}{l} \xi \mathrm{d}\xi\right)\cdot(-1)^k\cdot\frac{k\pi}{l}\\
&+\sum_{k=1}^{\infty}\left(\frac{2}{L-l} \int_{l}^{L}(\varphi(\eta)-\frac{L-\eta}{L-l}\sin \frac{k \pi}{L-l}(\eta-l)\mathrm{d} \eta\right)\cdot\frac{k\pi}{L-l}\\
&-\frac{1}{l}h(0)-\frac{1}{L-l}h(0)=0.
\end{aligned}
\end{equation}
It is found to be a necessary condition that the initial condition $u(x,0)=\varphi(x)$ satisfies \eqref{eq111} and such $\varphi(x)$ indeed exists, for example, $\varphi(x)=0$. Furthermore, for \eqref{eq111}, we immediately raise an additional question: Does this equation hold for a general  $C^2$ function $\varphi(x)$ with $\varphi(0)=\varphi(l)=\varphi(L)=0=\varphi''(0)=\varphi''(l)=\varphi''(L)$? This should be discussed in the future. In fact, we have compared the frequencies from formal solution (theoretical solution) corresponding to \eqref{eq2.3}-\eqref{eq2.2} obtained in this section with the frequencies from numerical solution of \eqref{eq1.1} fitted by the FEM method in engineering experiments. The discrepancy rate between theoretical values and numerical values is found to be less than 5.0\%, thereby cross-validating the accuracy of our newly established model \eqref{eq2.3}-\eqref{eq2.2} and the effectiveness of the engineering numerical simulations.

\subsection{PINNs for the wave equation with one elastic support in a classical sense}\label{subsec2.4}
\par
As mentioned above, when we view this problem through its intuitive physical form, the solution to the proposed mathematical model is well-posed in the classical sense. Therefore, theoretically, we can approximate the solution using appropriate numerical simulation methods \cite{MM2022}. Here, we adopt PINNs to initially construct a neural network for the proposed mathematical model.
\par
We use Python scripts to run the corresponding PINNs, supported by PyTorch. The descriptive scripts can be downloaded from \url{https://github.com/gracietan}. We describe the constructed PINNs $U$ below as simulated solution for the proposed mathematical model in the following basic architecture.
\par
In our numerical experiments, we specify the following physical parameters: take $L=70$, and the location of the elastic support is given at $l=17.5$. The first part of the string is defined on $D_{1}=[0,17.5]$ and the second part of the string is defined on $D_{2}=[17.5,70]$, and the wholeness is defined on $D=[0,70]$, with time interval $t\in[0,10]$. The wave speed is $a=67.344$, and $\sigma=0.005 \text{or} 1$. The initial conditions are given by $u(x,0)=\frac{1}{10}\sin(\frac{\pi}{35}x)$ and $u_{t}(x,0)=0$, with boundary conditions $u(0,t)=u(L,t)=0$.
\par
For the problem, we use a weighted loss function, i.e.,
\begin{equation}\label{eq2.73}
\begin{aligned}
\mathrm{Loss} =&\lambda_{\mathrm{PDE}}\mathrm{Loss}_{\mathrm{PDE}}+\lambda_{\mathrm{IC}}\mathrm{Loss}_{\mathrm{IC}}+\lambda_{\mathrm{BC}}\mathrm{Loss}_{\mathrm{BC}}+\lambda_{\mathrm{CC}}\mathrm{Loss}_{\mathrm{CC}}\\
=&\lambda_{\mathrm{PDE}}\mathrm{MSE}_{\mathrm{PDE}}+\lambda_{\mathrm{IC}}\mathrm{MSE}_{\mathrm{IC}}+\lambda_{\mathrm{BC}}\mathrm{MSE}_{\mathrm{BC}}+\lambda_{\mathrm{CC}}\mathrm{MSE}_{\mathrm{CC}}
\end{aligned}
\end{equation}
by assigning weights to each term of the loss function for $\mathrm{MSE}_{\mathrm{PDE}}$, $\mathrm{MSE}_{\mathrm{IC}}$, $\mathrm{MSE}_{\mathrm{BC}}$, together with $\mathrm{MSE}_{\mathrm{CC}}$ representing the mean squared errors corresponding to the equations, the initial conditions, the boundary conditions, and the constraint conditions respectively. More specifically, let
\begin{equation}\label{eq2.74}
\begin{aligned}
\mathrm{MSE}_{\mathrm{PDE}}=&\frac{1}{N_{\mathrm{PDE1}}} \sum\limits_{i=1}^{N_{\mathrm{PDE1}}}\left|U_{1tt}\left(x_{i}^{\mathrm{PDE1}}, t_{i}^{\mathrm{PDE1}}\right)-a^{2} U_{1xx}\left(x_{i}^{\mathrm{PDE1}}, t_{i}^{\mathrm{PDE1}}\right)\right|^{2} \\
+&\frac{1}{N_{\mathrm{PDE2}}} \sum\limits_{i=1}^{N_{\mathrm{PDE2}}}\left|U_{2tt}\left(x_{i}^{\mathrm{PDE2}}, t_{i}^{\mathrm{PDE2}}\right)-a^{2} U_{2 xx}\left(x_{i}^{\mathrm{PDE2}}, t_{i}^{\mathrm{PDE2}}\right)\right|^{2}. \\
\end{aligned}
\end{equation}
Here $\left(x_{i}^{\mathrm{PDE1}}, t_{i}^{\mathrm{PDE1}}\right) \in D_{1} \times[0, 10]$ and $\left(x_{i}^{\mathrm{PDE2}}, t_{i}^{\mathrm{PDE2}}\right) \in D_{2} \times[0, 10]$, and the number of uniform distributed training points for the equations are $N_{\mathrm{PDE1}}=200$ and $N_{\mathrm{PDE2}}=600$.
Similarly, let
\begin{equation}\label{eq2.75}
\begin{aligned}
\mathrm{MSE}_{\mathrm{IC}}=&\frac{1}{N_{\mathrm{IC1}}} \sum\limits_{i=1}^{N_{\mathrm{IC1}}}(|U_{1}(x_{i}^{\mathrm{IC1}},0)-\frac{1}{10}\sin(\frac{\pi}{35}x_{i}^{\mathrm{IC1}})|^{2}+|U_{1t}(x_{i}^{\mathrm{IC1}},0)|^{2})\\
+&\frac{1}{N_{\mathrm{IC2}}} \sum\limits_{i=1}^{N_{\mathrm{IC2}}}(|U_{2}(x_{i}^{\mathrm{IC2}},0)-\frac{1}{10}\sin(\frac{\pi}{35}x_{i}^{\mathrm{IC2}})|^{2}+|U_{2t}(x_{i}^{\mathrm{IC2}},0)|^{2}),
\end{aligned}
\end{equation}
\begin{equation}\label{eq2.76}
\mathrm{MSE}_{\mathrm{BC}}=\frac{1}{N_{\mathrm{BC}}} \sum\limits_{i=1}^{N_{\mathrm{BC}}}(\left|U_{1}(0, t_{i}^{\mathrm{BC}})\right|^{2}+\left|U_{2}(70, t_{i}^{\mathrm{BC}})\right|^{2}),
\end{equation}
\begin{equation}\label{eq2.77}
\begin{aligned}
{\mathrm{MSE}}_{\mathrm{CC}}=&\frac{1}{N_{\mathrm{CC_{1}}}} \sum\limits_{i=1}^{N_{\mathrm{CC1}}}|U_{1}(17.5, t_{i}^{\mathrm{CC_{1}}})-U_{2}(17.5, t_{i}^{\mathrm{CC_{1}}})|^{2} \\
+&\frac{1}{N_{\mathrm{CC_{2}}}} \sum_{i=1}^{N_{\mathrm{CC_{2}}}}|-U_{1x}(17.5, t_{i}^{\mathrm{CC_{2}}})+U_{2x}(17.5, t_{i}^{\mathrm{CC_{2}}})-\sigma U_{1}(17.5,t_{i}^{\mathrm{CC_{2}}})|^{2}, \\
\end{aligned}
\end{equation}
where $x_{i}^{\mathrm{IC1}}\in D_1, x_{i}^{\mathrm{IC2}}\in D_2$, $t_{i}^{\mathrm{BC}} \in[0, 10]$, $t_{i}^{\mathrm{CC_{1}}} \in[0, 10]$, and $t_{i}^{\mathrm{CC_{2}}} \in[0, 10]$, together with $N_{\mathrm{IC1}}=200, N_{\mathrm{IC2}}=600$ as the number of uniform distributed training points for the initial conditions, $N_{\mathrm{BC}}=100$ as the number of uniform distributed training points for the boundary conditions, and $N_{\mathrm{CC1}}=200$, $N_{\mathrm{CC2}}=600$ as the number of uniform distributed training points for the constraint conditions \eqref{eq2.3}-\eqref{eq2.2}.

\begin{figure}[htbp]
	\centering
	\begin{minipage}[t]{0.49\textwidth}
		\centering
		\includegraphics[width=8.2cm]{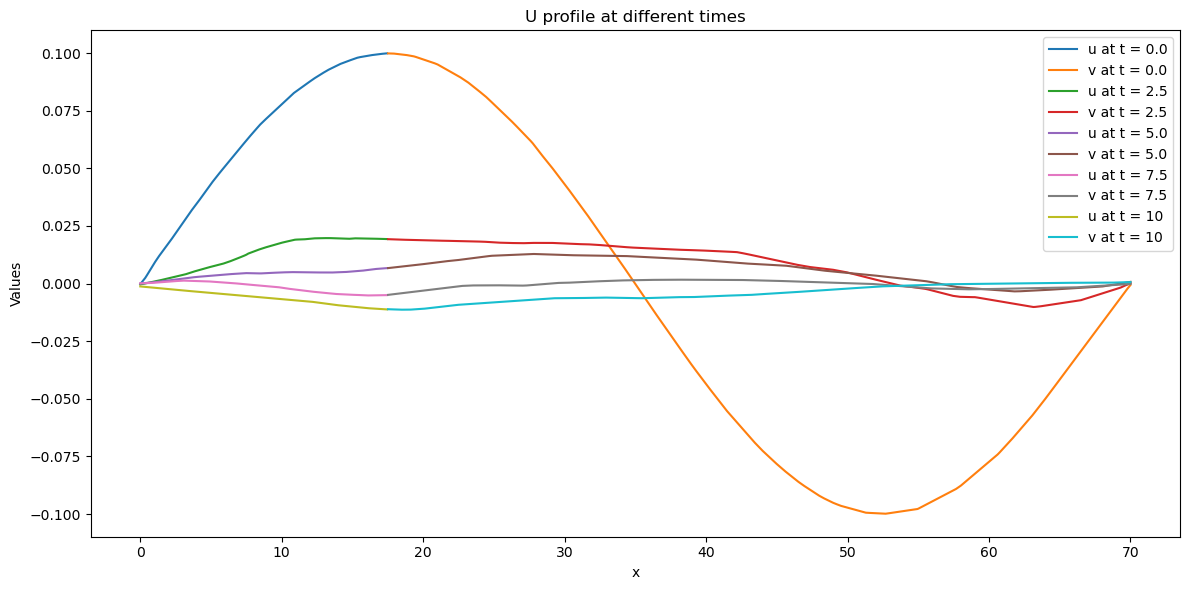}
		\caption{The profile of $U$ at different time with an elastic support at $x=17.5$ when $\sigma=0.005$\label{fig21}}
	\end{minipage}
	\begin{minipage}[t]{0.49\textwidth}
		\centering
		\includegraphics[width=8.2cm]{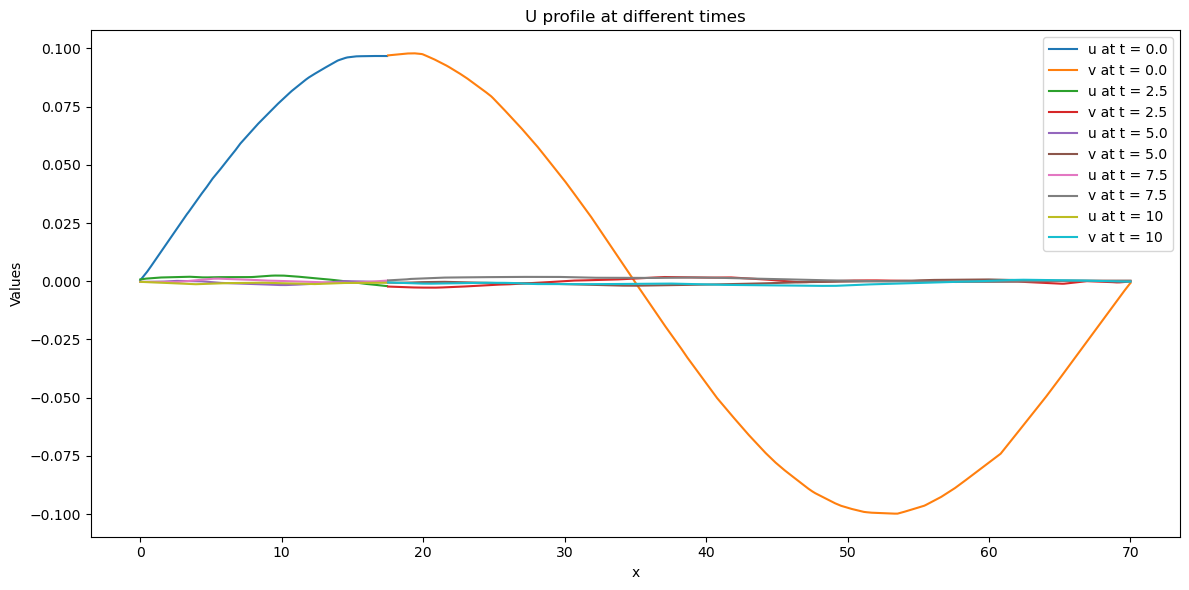}
		\caption{The profile of $U$ at different time with an elastic support at $x=17.5$ when $\sigma=1$\label{fig22}}
	\end{minipage}
\end{figure}

\begin{figure}[htbp]
	\centering
	\begin{minipage}[t]{0.49\textwidth}
		\centering
		\includegraphics[width=6.8cm]{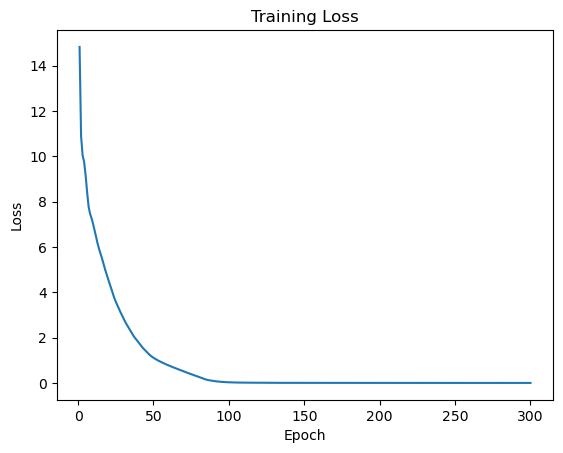}
		\caption{Training loss over epochs when $\sigma=0.005$\label{fig23}}
	\end{minipage}
	\begin{minipage}[t]{0.49\textwidth}
		\centering
		\includegraphics[width=6.8cm]{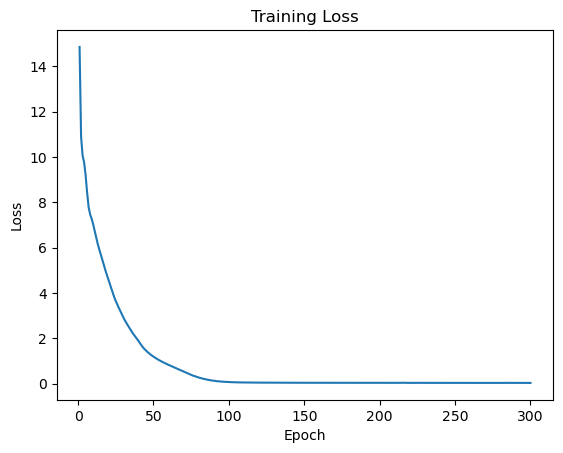}
		\caption{Training loss over epochs when $\sigma=1$\label{fig24}}
	\end{minipage}
\end{figure}
In Fig.\ref{fig21} and Fig.\ref{fig22}, we show the model fitting of the classical solution through the method of PINNs and plot figures to illustrate how the solution changes over time. Moreover, we show the training loss over epochs in Fig.\ref{fig23} and Fig.\ref{fig24}. To some extent, it can be clearly seen that elastic support has a stabilizing effect, and initial disturbances are quickly subdued, which meets the engineering requirements.
\begin{remark}
We notice that in this numerical experiment, the given initial value $u(x,0)=\varphi(x)\ne0$ at $x=l$, namely it does not satisfy the necessary condition mentioned in Theorem \ref{thm2.3}. Therefore, the simulated solution $U$ of PINNs on $[0,L]$ is not a $C^2$ solution, but at most a piecewise $C^2$ solution according to the position of known elastic support.
\end{remark}
\section{The wave equation with elastic supports in a weak sense}\label{sec3}
\par
Now, we would like to solve the wave equation with a finite number of elastic supports in terms of weak solution since a global classical solution is not obtained. For clarity, the initial-boundary value problem \eqref{eq1.1} is relabeled here as
\begin{equation}\label{eq3.1}
\left\{\begin{array}{l}
u_{t t}-a^{2} u_{x x}=\sum\limits_{k=1}^{n}-\beta_k u\left(x_{k}, t\right) \delta_{x=x_k},  (x, t) \in(0, L) \times(0, T], \\
u(0,t)=0, \quad u(L,t)=0,  t\in(0, T], \\
u(x,0)=\varphi(x), \quad u_{t}(x,0)=\psi(x),  x\in[0, L],
\end{array}\right.
\end{equation}
where $a$, $T$, $L$ and $\beta_k$ are positive constants with $x_k\in(0,L) (k=1,2,\dots,n)$.

\subsection{Definition of weak solution}\label{sec3.1}
First of all, we should focus on how to define weak solution corresponding to the initial-boundary value problem \eqref{eq3.1}. In order to do this, $u$ is assumed as a smooth solution to \eqref{eq3.1} and we respectively use $u'$ and $u''$ to represent the first and second derivative of the function $u$ with respect to $t$. We also denote $(u,v)$ as the inner product of $u$ and $v$ in the space $L^2\left((0,L)\right)$. For any $v\in H_{0}^{1}\left((0,L)\right)$, we multiply \eqref{eq3.1} by $v$, integrate it over $(0,L)$ and obtain
\begin{equation}\label{eq3.2}
\int_{0}^{L} \left(u_{t t} v-a^{2} u_{x x}v+\sum_{k=1}^{n} \beta_{k} u\left(x_{k}, t\right) v\left(x_{k}, t\right)\right)\mathrm{d}x =0.
\end{equation}
Let us introduce
\begin{equation}\label{eq112}
B[u, v ; t] \doteq \int_{0}^{L} a^{2} u_{x}(x,t) v_{x}(x,t) \mathrm{d}x+\sum\limits_{k=1}^{n} L \beta_{k} u\left(x_{k}, t\right) v\left(x_{k}, t\right).
\end{equation}
as a real bilinear form. Now, it is worth noting that \eqref{eq3.2} can be written as
\begin{equation}\label{eq3.3}
(u'', v) + B[u, v; t] = 0, 0 \leq t \leq T.
\end{equation}
Because $v \in H_{0}^{1}((0,L))$, as long as $u(\cdot,t)\in H_{0}^{1}((0,L))$ and $u_{tt}(\cdot,t)\in H^{-1}((0,L))$, then \eqref{eq3.3} makes sense. Hereafter, $\langle \cdot, \cdot \rangle$ is used to represent the duality pairing between $H^{-1}((0,L))$ and $H_{0}^{1}((0,L))$.
\begin{lemma}\label{lem3.1}
$B[u, v ; t]$ is claimed to be a real bilinear form as a result of\\
(1) $B[u, v; t]$ is bounded: there exists a positive constant $C_{1}$, such that
\begin{center}
$|B[u, v; t]| \leq C_{1}\|u(\cdot,t)\|_{H_{0}^{1}\left((0,L)\right)}\|v(\cdot,t)\|_{H_{0}^{1}((0, L))}, \forall u(\cdot,t), v(\cdot,t)\in H_{0}^{1}((0, L))$;
\end{center}
(2) $B[u, u; t]$ is coercive: there exists a positive constant $C_{2}$, such that
\begin{center}
$B[u, u ; t] \geqslant C_{2}\|u(\cdot,t)\|_{H_{0}^{1}((0, L))}^{2}, \forall u(\cdot,t)\in H_{0}^{1}((0, L))$.
\end{center}
\end{lemma}
\begin{proof}
(1) Since $a^2, \beta_{k}$ are bounded and $\beta_{k}\ge0$, there are positive constants $C_1$ and $C$, such that
\begin{center}
$\begin{aligned}
\left|B\left[u, v;t\right]\right|& = \left|\int_{0}^{L} a^{2} u_{x} v_{x} \mathrm{d}x + \sum_{k=1}^{n}L\beta_{k}u\left(x_k, t\right) v\left(x_k, t\right)\right|\\
&\leq C_1\|\mathrm{D}u(\cdot,t)\|_{L^{2}((0, L))}\|\mathrm{D}v(\cdot,t)\|_{L^{2}((0, L))} + C_1\|u(\cdot,t)\|_{L^{\infty}((0, L))}\|v(\cdot,t)\|_{L^{\infty}((0, L))} \\
&\leq C\|u(\cdot,t)\|_{H_{0}^{1}((0, L))}\|v(\cdot,t)\|_{H_{0}^{1}((0, L))}.
\end{aligned}$
\end{center}
(2) Since $a^2, \beta_{k}$ are bounded and $\beta_{k}\ge0$, by the uniform hyperbolicity condition, there is a positive constant $C_{2}$, such that
\begin{center}
$\begin{aligned}
B\left[u, u; t\right] &= \int_{0}^{L} a^{2} u_{x}^{2} \mathrm{d}x + \sum_{k=1}^{n}L\beta_{k} u^{2}(x_k, t)\\
& \geq a^{2}\int_{0}^{L}|\mathrm{D}u|^{2}\mathrm{d}x\\
& \geq C_{2}\|u(\cdot,t)\|_{H_{0}^{1}((0, L))}^{2}.
\end{aligned}$
\end{center}
\end{proof}
\begin{definition}\label{dy3.1}
If $u \in L^{2}(0, T; H_{0}^{1}((0, L)))$ satisfying\\
(H1)  $u'\in L^{2}\left(0, T ; L^2((0, L)))\right., u''\in L^{2}\left(0, T ; H^{-1}((0, L)))\right.$;\\
(H2) For each $v\in H_{0}^{1}((0, L))$,
\begin{center}
$\left\langle u'', v\right\rangle+B[u, v ; t]=0$, for $t \in [0, T]$ a.e.;\\
\end{center}
(H3) $u(0)=\varphi(x) \in H_{0}^{1}((0,L)),u'(0)=\psi(x) \in L^2((0,L))$,\\
then the function $u$ is called a weak solution to the initial-boundary value problem \eqref{eq3.1} of the wave equation with a finite number of elastic supports.
\end{definition}
Ultimately, we obtain the definition of the solution to \eqref{eq3.1}, as one of the main results of our work.
\subsection{Existence and uniqueness of weak solution}\label{sec3.2}
\par
Here, we state another main result of this paper.
\begin{theorem}\label{thm3.1}
Under the assumptions (H1), (H2) and (H3), \eqref{eq3.1} admits a unique solution $u \in L^{\infty}\left(0, T ; H_{0}^{1}((0, L))\right)$.
\end{theorem}
\begin{proof}
If $u \in L^{\infty}(0, T ; H_{0}^{1}((0, L)))$, 
we have $ u_{x x} \in L^{\infty}\big(0, T; H^{-1}((0, L))\big)$, 
and $u''\in L^{2} (0, T ; H^{-1}((0, L)))$. 
Furthermore, 
if $u'\in L^{\infty}(0 , T ; L ^ { 2 } ( ( 0 , L ) ))$, 
$u \in C ([0, T] ; L^{2}((0, L)))$ 
and $u''\in C([0, T] ; H^{-1}((0, L)))$ are obtained. 
Therefore, $u'(0)$ and $u(0)$ make sense \cite{Evans1998Partial}.

Taking complete smooth functions $\left\{w_{k}=w_{k}(x)\right\}(k=1,2, \cdots)$ which are orthogonal both in $L^2((0,L))$ and $H_{0}^{1}((0,L))$, one can find
\begin{center}
$\left\{w_{k}\right\}_{k=1}^{\infty}$ is the orthogonal basis for $H_{0}^{1}((0, L))$
\end{center}
and
\begin{center}
$\quad\left\{w_{k}\right\}_{k=1}^{\infty}$ is the standard orthogonal basis for $L^{2}((0,L))$ also in \cite{Evans1998Partial}.
\end{center}
Fix an integer $m$ and consider $\left\{w_{1},w_{2},\dots,w_{m}\right\}$ as $m$-linearly independent functions in $H_{0}^{1}((0,$\\$L))$, which generates an $m$-dimensional linear subspace $W_{m}((0, L))=\operatorname{span}\{w_{1},\dots,w_{m}\}$. Moreover, we choose $\xi_{m, i}$ and $\eta_{m,j}$ here so that when $m \rightarrow \infty$,
\begin{equation}\label{eq3.4}
\begin{array}{l}
\left\|\sum\limits_{i=1}^{m} \xi_{m,i} w_{i}-\varphi(x)\right\|_{H_{0}^{1}((0, L))}{\longrightarrow} 0,   \left\|\sum\limits_{i=1}^{m} \eta_{m,i} w_{i}-\psi(x)\right\|_{L^{2}((0, L))}{\longrightarrow} 0.\\
\end{array}
\end{equation}
By virtue of finding an approximate $m$-order solution $u_m$ on $W_{m}((0, L))$, we have
\begin{center}
$u_{m}=\sum\limits_{i=1}^{m} d_{m,i}(t) w_{i}(x)$.
\end{center}
When taking $v=\left\{w_{j}\right\}_{j=1}^{\infty} \in H_{0}^{1}((0, L))$, and $u=u_{m}$, we then have
\begin{equation}\label{eq3.5}
\left(u_{m}^{\prime \prime}, w_{j}\right)+B\left[u_{m}, w_{j} ; t\right]=0   \text {a.e.} 0 \leq t \leq T,
\end{equation}
where
\begin{equation}\label{eq3.6}
\left(u_{m}^{\prime \prime}, w_{j}\right)=\left(\sum_{i=1}^{m} d_{m,i}^{\prime \prime}(t) w_{i}(x), w_{j}(x)\right)=d_{m,j}^{\prime \prime}(t),
\end{equation}
together with
\begin{equation}\label{eq3.7}
\begin{aligned}
B\left[u_{m}, w_{j} ; t\right] &=\int_{0}^{L} a^{2} (u_m)_x(w_j)_x\mathrm{d}x+\sum\limits_{k=1}^{n}L\beta_k u_{m}\left(x_k, t\right) w_{j}\left(x_k\right) \\
&=\int_{0}^{L} a^{2} \sum\limits_{i=1}^{m} d_{m,i}(t) (w_{i}(x))_{x}(w_{j}(x))_{x} \mathrm{d} x+\sum\limits_{k=1}^{n}\sum\limits_{i=1}^{m}L\beta_{k} d_{m,i}(t) w_{i}\left(x_k\right) w_{j}\left(x_k\right)\\
&=\sum\limits_{i=1}^{m} B\left[w_{i}, w_{j} ; t\right] d_{m,i}(t).
\end{aligned}
\end{equation}
Consequently, it yields that
\begin{equation}\label{eq3.8}
d_{m, j}^{\prime \prime}(t)+\sum_{i=1}^{m} B\left[w_{i}, w_{j};t\right] d_{m,i}(t)=0,
\end{equation}
and it can be seen that $d_{m,i}=d_{m,i}(t)$ is determined by the Cauchy problem of the following system of second-order ordinary differential equations with constant coefficients
\begin{equation}\label{eq3.9}
\left\{\begin{array}{l}
d_{m,j}^{\prime \prime}(t)+\sum\limits_{i=1}^{m} B\left[w_{i}, w_{j};t\right] d_{m,i}(t)=0, \\
d_{m,j}(0)=\xi_{m,j},  d_{m, j}^{\prime}(0)=\eta_{m,j},  j=1, \dots, m.
\end{array}\right.
\end{equation}
Since the coefficients of $d_{m,i}$ are constants, it is known that there exists a unique solution $d_{m,i}(t)$, $i=1,\dots,m$ in \eqref{eq3.9} by the existence and uniqueness theorem in ordinary differential equations, with $d_{m,i} \in C^{2}([0, T])$ and $d''_{m,i}$ being absolutely continuous over $[0,T]$. Thus, an approximate solution $u_m$ is obtained.
\par
With multiplying \eqref{eq3.5} by $d'_{m,j}$, summing $j$ from $1$ to $m$ and noting the expression for $u_m$, we integrate \eqref{eq3.5} over $(0,t)$ to conclude
\begin{equation}\label{eq3.10}
\int_{0}^{t}((u_{m}^{\prime \prime}, u_{m}^{\prime})+B\left[u_{m}, u_{m}^{\prime} ; t\right])\mathrm{d} t=0,  0 \leq t \leq T.
\end{equation}
It is obvious that
\begin{equation}\label{eq3.11}
\int_{0}^{t}\left(u_{m}^{\prime \prime}, u_{m}^{\prime}\right) \mathrm{d} t=\frac{1}{2} \int_{0}^{t}\int_{0}^{L}\frac{\mathrm{d}}{\mathrm{d} t}\left(u_{m}^{\prime}\right)^{2} \mathrm{d} x\mathrm{d} t=\frac{1}{2}\left\|u_{m}^{\prime}(x,t)\right\|_{L^{2}(( 0, L))}^{2}-\frac{1}{2}\left\|u_{m}^{\prime}(x,0)\right\|_{L^{2}(( 0, L))}^{2}.
\end{equation}
In addition,
\begin{equation}\label{eq3.12}
\begin{aligned}
B\left[u_{m}, u_{m}^{\prime} ; t\right]&=\int_{0}^{L} a^{2} (u_m)_{x}(u_{m}^{\prime})_{x}\mathrm{d}x+\sum\limits_{k=1}^{n}L\beta_{k} u_{m}\left(x_k, t\right) u_{m}^{\prime}(x_k, t)\\
&=\frac{\mathrm{d}}{\mathrm{d} t} \left( \int_{0}^{L} \frac{1}{2} a^{2} (u_m)_{x}^{2}\mathrm{d} x+\frac{1}{2}\sum\limits_{k=1}^{n}L\beta_{k}u_{m}^{2}(x_k, t)\right).
\end{aligned}
\end{equation}
Thus, we have a positive constant $C_2$, such that
\begin{equation}\label{eq3.13}
\begin{aligned}
\int_{0}^{t} B\left[u_{m}, u_{m}^{\prime} ; t\right] \mathrm{d}t
&=\int_{0}^{L} \frac{1}{2} a^{2} (u_{m}(x, t))_{x}^{2}\mathrm{d}x-\int_{0}^{L} \frac{1}{2} a^{2} (u_{m}(x, 0))_{x}^{2}\mathrm{d}x\\
&                                          +\frac{1}{2}\sum\limits_{k=1}^{n}L\beta_{k}u_{m}^{2}(x_k,t)-\frac{1}{2}\sum\limits_{k=1}^{n}L\beta_{k}u_{m}^{2}(x_k,0)\\
&\geqslant C_{2}\left(\left\|u_{m}(x,t)\right\|_{H_{0}^{1}((0, L))}^{2}-\int_{0}^{L}(u_{m}(x, 0))_{x}^{2}\mathrm{d}x+u_{m}^{2}(x_k,t)-u_{m}^{2}(x_k,0)\right).
\end{aligned}
\end{equation}
We substitute \eqref{eq3.11}, \eqref{eq3.13} into \eqref{eq3.10} to conclude
\begin{equation}\label{eq3.14}
\begin{array}{l}
\frac{1}{2}\|u_{m}^{\prime}(x,t)\|_{L^{2}((0, L))}^{2}-\frac{1}{2}\|u_{m}^{\prime}(x,0)\|_{L^{2}((0, L))}^{2}+C_{2}\left\|u_{m}(x,t)\right\|_{H_{0}^{1}((0, L))}^{2}\\
-C_{2}\int_{0}^{L}(u_{m}(x, 0))_{x}^{2}\mathrm{d}x+C_{2}u_{m}^{2}(x_k,t)-C_2u_{m}^{2}(x_k,0)\leq 0.
\end{array}
\end{equation}
Therefore, there are positive constants $M$ and $C$ being independent of $m$, such that
\begin{equation}\label{eq3.15}
\begin{array}{l}
\left\|u_{m}^{\prime}(x,t)\right\|_{L^{2}((0, L))}^{2}+\left\|u_{m}(x,t)\right\|_{H_{0}^{1}((0, L))}^{2} \\
\leq M\left(\|u_m^{\prime}(x,0)\|_{L^2((0,L))}^{2}+\left\|u_{m}(x,0)\right\|_{H_{0}^{1}((0, L))}^{2}\right)\\
\le C.
\end{array}
\end{equation}
Because of the weak compactness, there is a sequence still denoted as $u_m$ and when $m{\rightarrow}\infty$, there is  $u \in L^{2}\left(0, T ; H_{0}^{1}((0, L))\right), u^{\prime} \in L^{2}\left(0, T ; L^{2}((0, L))\right)$ satisfying
\begin{center}
$\begin{aligned}
u_{m}\rightharpoonup u \text{in} L^{2}\left(0, T ; H_{0}^{1}((0, L))\right) \\
u_{m}^{\prime} \rightharpoonup u^{\prime} \text{in} L^{2}\left(0, T ; L^{2}((0, L))\right).
\end{aligned}$
\end{center}
Set
\begin{center}
$C_{T}^{1}([0, T])=\left\{\phi \in C^{1}([0, T]) \mid \phi(T)=0\right\}$.
\end{center}
Since $\left\{w_{k}\right\}_{k=1}^{\infty}$ is an orthogonal basis of $H_{0}^{1}((0, L))$, any element belonging to $H_{0}^{1}((0, L))$ can be represented linearly by $w_{k}$. For the fixed $m>N$ and any given $d_{j} \in C_{T}^{1}([0, T])$,we consider the function
\begin{equation}\label{eq3.16}
\omega=\sum_{j=1}^{N} d_{j} w_{j}.
\end{equation}
Multiplying \eqref{eq3.5} by $d_j$, summing from $1$ to $N$, and integrating the result by parts with respect to $t$, we have
\begin{equation}\label{eq3.17}
\int_{0}^{T}((u_{m}^{\prime \prime}, \omega)+B\left[u_{m}, \omega ; t\right]) \mathrm{d} t=0,
\end{equation}
where
\begin{equation}\label{eq3.18}
\begin{aligned}
\int_{0}^{T}\left(u_{m}^{\prime \prime}, \omega\right) \mathrm{d} t&=\int_{0}^{T}(\frac{\mathrm{d}}{\mathrm{d} t} u_{m}^{\prime}, \omega) \mathrm{d} t\\
&=-\int_{0}^{T}\left(u_{m}^{\prime}, \omega^{\prime}\right) \mathrm{d} t-\left(u_{m}^{\prime}(0), \omega(0)\right).
\end{aligned}
\end{equation}
As a result, we have
\begin{equation}\label{eq3.19}
\int_{0}^{T}(B\left[u_{m}, \omega; t\right]-(u_{m}^{\prime}, \omega^{\prime})) \mathrm{d} t=(u_{m}^{\prime}(0), \omega(0)).
\end{equation}
Letting $m \rightarrow \infty$, we obtain
\begin{center}
$\int_{0}^{T}(B[u, \omega, t]-(u^{\prime}, \omega^{\prime})) \mathrm{d} t=(u^{\prime}(0), \omega(0)).$
\end{center}
Also, let
\begin{center}
$\Omega(x,t)=\left\{\omega \in L^{2}\left(0, T ; H_{0}^{1}((0, L))\right) \mid \omega^{\prime} \in L^{2}\left(0, T ; L^{2}((0, L))\right), \omega(x, T)=0\right.,x \in(0, L)\}.$
\end{center}
The norm of $\Omega$ is taken here by
\begin{center}
$\|\omega\|_{\Omega}^{2}=\|\omega\|_{L^{2}((0, T ; H_{0}^{1}((0, L)))}^{2}+\left\|\omega^{\prime}\right\|_{\left.L^{2}(0, T; L^{2}((0, L))\right)}^{2}$,
\end{center}
and the function given by \eqref{eq3.19} is dense in $\Omega$. Then, for any $\omega\in \Omega$, we have
\begin{equation}\label{eq3.20}
\int_{0}^{T}(B[u, \omega ; t]-(u^{\prime}, \omega^{\prime})) \mathrm{d} t=(u^{\prime}(0), \omega(0)).
\end{equation}
For any $\phi \in C_{0}^{\infty}([0, T])$, and $v\in H_{0}^{1}((0, L))$, insert $\omega=\phi v$ into \eqref{eq3.20}, which yields
\begin{center}
$\left\langle\int_{0}^{T}(a^{2} u_{x}\phi_{x}-\left(u^{\prime}, \phi^{\prime}\right)+\sum\limits_{k=1}^{n}L\beta_{k} u\left(x_k, t\right)\phi(x_k,t))\mathrm{d}t, v\right\rangle=0, \forall v \in H_{0}^{1}((0, L)).$
\end{center}

Therefore,
\begin{center}
$\int_{0}^{T}(a^{2} u_{x}\phi_{x}-\left(u^{\prime}, \phi^{\prime}\right)+\sum\limits_{k=1}^{n}L\beta_{k} u(x_k, t)\phi(x_k,t))\mathrm{d}t=0, \forall \phi \in C_{0}^{\infty}([0 ,T])$.
\end{center}
It means that $u^{\prime \prime} \in L^{2}(0,T ; H^{-1}((0, L)))$ and for any $v\in H_{0}^{1}((0, L))$, we have
\begin{center}
$\left\langle u^{\prime \prime}, v\right\rangle+B[u, v ; t]=0$ a.e. $t\in[0,T]$.
\end{center}
This verifies (H1) and (H2) in the definition of the weak solution.
\par
It is shown below that $u$ satisfies the initial conditions of (H3) in the definition of the weak solution, i.e. Definition \ref{dy3.1}.
\par
We take $\omega \in C^{2}\left([0, T] ; H_{0}^{1}((0, L))\right)$ which satisfies $\omega(x, T)=\omega^{\prime}(x, T)=0, x \in(0, L)$ in \eqref{eq3.20}\\and again use integration by part with respect to $t$ to get
\begin{equation}\label{eq3.21}
\begin{aligned}
&\int_{0}^{T}(\left\langle\omega^{\prime \prime}, u\right\rangle+B[u, \omega ; t]) \mathrm{d} t\\
&=(u^{\prime}(0), \omega(0))+(u(T), \omega^{\prime}(T))-(u(0), \omega^{\prime}(0))\\
&=(u^{\prime}(0), \omega(0))-(u(0), \omega^{\prime}(0)) .
\end{aligned}
\end{equation}
Similarly, by \eqref{eq3.19}, we get
\begin{center}
$\int_{0}^{T}(\left\langle\omega^{\prime \prime}, u_{m}\right\rangle+B\left[u_{m}, \omega; t\right]) \mathrm{d} t=
(u_{m}^{\prime}(0), \omega(0))-(u_{m}(0), \omega^{\prime}(0)).$
\end{center}
For the subsequence $u_{m_{l}}$, using the initial conditions of the problem $\eqref{eq3.9}_2$, we can conclude that as $m_l \rightarrow \infty$,
\begin{equation}\label{eq3.22}
\begin{aligned}
\left.u_{m_{l}}\right|_{t=0} & =\sum_{i=1}^{m_{l}} d_{m_{l, i}}(0)w_{i}(x) =\sum_{i=1}^{m_l} \xi_{m_{l, i}} w_{i}(x) \longrightarrow \varphi(x),\\
\left.u_{m_{l}}^{\prime}\right|_{t=0} & =\sum_{i=1}^{m_{l}} d'_{m_{l, i}}(0) w_{i}(x)=\sum_{i=1}^{m_{l}} \eta_{m_l,i} w_{i}(x) \longrightarrow \psi(x),
\end{aligned}
\end{equation}
so
\begin{equation}\label{eq3.23}
\int_{0}^{T}(\left\langle\omega^{\prime \prime}, u\right\rangle+B[u, \omega;t]) \mathrm{d} t=(u_{1}, \omega(0))-(u_{0}, \omega^{\prime}(0)).
\end{equation}
If we compare \eqref{eq3.21} with $\eqref{eq3.23}$ and use the arbitrariness of $\omega(0)$ and $\omega'(0)$, it is known that
\begin{center}
$u(0)=\varphi(x),   u'(0)=\psi(x).$
\end{center}
So far, we have proved the existence of the weak solution.
\par
We are going to prove the uniqueness of the weak solution. This is equivalent to proving that if $\varphi(x)=\psi(x)=0$, then $u=0$. For any $s\in(0,T)$, we construct a test function as
\begin{center}
$\omega(x,t)=\left\{\begin{array}{ll}
-\int_{t}^{s} u(x,\nu ) d \nu, & t \leq s, \\
0, & t>s .
\end{array}\right.$
\end{center}
It is known that when $0\le t\le s$, $\omega'=u(x,t)$. $\omega$ is taken into \eqref{eq3.19}, yielding
\begin{center}
$\int_{0}^{s}(B\left[\omega^{\prime}, \omega ; t\right]-(u^{\prime}, u)) \mathrm{d} t=0,$
\end{center}
where
\begin{center}
$\begin{aligned}
B\left[\omega^{\prime}, \omega ; t\right] &=\int_{0}^{L} a^{2} \omega_{x}^{\prime} \omega_{x}\mathrm{d}x+\sum\limits_{k=1}^{n}L\beta_k\omega^{\prime}(x_k, t) \omega(x_k, t)\\
&=\frac{1}{2} \frac{\mathrm{d}}{\mathrm{d} t} B[\omega, \omega; t],
\end{aligned}$
\end{center}
thus,
\begin{center}
$\int_{0}^{s}(\frac{1}{2} \frac{\mathrm{d}}{\mathrm{d} t} B[\omega, \omega ; t]-\frac{1}{2} \frac{\mathrm{d}}{\mathrm{d} t}(u, u))\mathrm{d}t=0$,
\end{center}
i.e.
\begin{center}
$B[\omega(x,0), \omega(x,0) ; 0]+(u(x,s), u(x,s))=0$.
\end{center}
It is known that there are positive constants $C_3$ and $C$, so that
\begin{equation}\label{eq3.24}
C_{3}\int_{0}^{L}|\mathrm{D} \omega{2} \mathrm{d} x \leqslant \int_{0}^{L} a^{2} \omega(0)_{x} \omega(0)_{x} \mathrm{d} x=B[\omega(0), \omega(0) ; 0].
\end{equation}
Notice that
\begin{equation}\label{eq3.25}
C_{3}\int_{0}^{L}|\mathrm{D} \omega(x,0)|^{2} \mathrm{d} x\geqslant C\|\omega(\cdot,0)\|_{H_{0}^{1}((0, L))}^{2},
\end{equation}
then
\begin{equation}\label{eq3.26}
\|\omega(\cdot,0)\|_{H_{0}^{1}((0, L))}^{2}+\|u(\cdot,s)\|_{L^{2}((0, L))}^{2} \leqslant 0.
\end{equation}
Hence $u=0$.
\end{proof}
\subsection{Regularity of weak solution}\label{subsec3.3}
\begin{theorem}\label{thm3.2}
Along with the assumptions of Theorem \ref{thm3.1} and $\varphi(x) \in H^{2}((0, L)), \psi(x) \in H_{0}^{1}((0, L))$, there exists a unique $u \in C\left([0, T] ; H_{0}^{1}((0, L))\right) \cap L^{\infty}\left(0, T ; H^{2}((0, L))\right)$ that satisfies $u^{\prime} \in C\left([0, T];L^{2}((0, L))\cap H_{0}^{1}((0, L)) \right)$, $u^{\prime \prime} \in L^{\infty}\left(0, T ; L^{2}((0, L))\right)$ and \eqref{eq3.1}.
\end{theorem}
\begin{proof}
Now, under these assumptions, we may assume $w_1 = \varphi(x)$, and let the approximate solution $u_m = \sum\limits_{i=1}^{m} d_{m,i}(t) w_i(x)$ satisfy
\begin{equation}\label{eq3.27}
\left\{\begin{array}{l}
\left(u_{m}^{\prime \prime}, w_{j}\right)+B\left[u_{m,} w_{j} ; t\right]=0, j=1,\dots,m, \\
d_{m, 1}(0)=1,  d_{m_{j}}(0)=0, j=2,\dots, m, \\
d_{m, j}^{\prime}(0)=\eta_{m, i}, m \rightarrow \infty,  \sum\limits_{i=1}^{m} \eta_{m, i} w_{i}=\psi(x) \left(\text{in} H_{0}^{1}((0, L))\right)
\end{array}\right.
\end{equation}
According to Theorem \ref{thm3.1}, the Cauchy problem for the system of ordinary differential equations \eqref{eq3.27} has a unique solution $\{d_{m, i}\}_{i=1}^{m}$, and the corresponding approximate solution $u_m$ has already been estimated as
\begin{equation}\label{eq3.28}
\left\|u_{m}^{\prime}(t)\right\|_{L^{2}((0, L))}^{2}+\left\|u_{m}(x, t)\right\|_{H_{0}^{1}((0,L))}^{2} \leq C\left(\left\|\varphi(x)\right\|_{H_{0}^{1}((0, L))}^{2}+\left\|\psi(x)\right\|_{L^{2}((0, L))}^{2}\right),
\end{equation}
in which $C$ is a positive constant. Now, differentiate both sides of \eqref{eq3.27} with respect to $t$, and denote $\tilde{u}_{m}=u_{m}^{\prime}$, then we have
\begin{equation}\label{eq3.29}
\left(\tilde{u}_{m}^{\prime \prime}, w_{j}\right)+B\left[\tilde{u}_{m}, w_{j} ; t\right]=0, j=1, \dots, m.
\end{equation}
Multiplying \eqref{eq3.29} by $d_{m, j}^{\prime \prime}$, summing over $j$ from $1$ to $m$, and then integrating it over $[0,t]$ with respect to $t$, we have
\begin{equation}\label{eq3.30}
\int_{0}^{t}((\tilde{u}_{m}^{\prime \prime}, \tilde{u}_{m}^{\prime})+B\left[\tilde{u}_{m}, \tilde{u}_{m}^{\prime} ; t\right]) \mathrm{d} t=0.
\end{equation}
Note that this equation has the same structure as \eqref{eq3.10}. Therefore, following the derivation of \eqref{eq3.15}, it is known that there exist positive constant $C$ being independent of $m$, such that
\begin{equation}\label{eq3.32}
\begin{array}{l}
\|u_{m}^{\prime \prime}(\cdot, t)\|_{L^{2}((0, L))}^{2}+\|u_{m}^{\prime}(\cdot, t)\|_{H_{0}^{1}((0, L))}^{2}\leqslant C \| u_{m}^{\prime \prime}(\cdot, 0)\|_{L^{2}((0, L))}^{2}+C\| u_{m}^{\prime}(\cdot, 0) \|_{H_{0}^{1}((0, L))}^{2}.
\end{array}
\end{equation}
\par
In fact, as $m\rightarrow \infty$, $u_{m}^{\prime}(x, 0) \rightarrow \psi(x)$ in $H_{0}^{1}((0, L))$, so there exists a positive constant $C_{4}$ being independent of $m$, such that
\begin{equation}\label{eq3.33}
\|u_{m}^{\prime}(x,0)\|_{H_{0}^{1}((0,L))}^{2}\leq C_{4}\|\psi(x)\|_{H_{0}^{1}((0,L))}^{2}.
\end{equation}
Secondly, we should estimate $\left\|u_{m}^{\prime \prime}(x, 0)\right\|_{L^{2}((0, L))}^{2}$. By taking $t=0$ in \eqref{eq3.27}, multiplying it by $d_{m,j}^{\prime \prime}(0)$, and summing the result over $j$ from $1$ to $m$, we have
\begin{equation}\label{eq3.34}
\left(u_{m}^{\prime \prime}(x, 0), u_{m}^{\prime \prime}(x, 0)\right)+B\left[u_{m}(x, 0), u_{m}^{\prime \prime}(x, 0) ; 0\right]=0.
\end{equation}
Then there exist positive constants $C_{5}$ and $C$,
\begin{equation}\label{eq3.35}
\begin{aligned}
\left\|u_{m}^{\prime \prime}(x, 0)\right\|_{L^{2}((0, L))}^{2}&=-B\left[u_{m}(x, 0), u_{m}^{\prime \prime}(x, 0) ; 0\right] \\
&=\int_{0}^{L} a^{2}(u_{m}(x, 0))_{xx} (u_{m}^{\prime \prime}(x, 0))\mathrm{d}x-\sum\limits_{k=1}^{n}L\beta_ku_{m}(x_k, 0) u_{m}^{\prime \prime}(x_k, 0)\\
&\le C_{5}\|\mathrm{D}^2u_{m}(x,0)\|_{L^2((0,L))}^{2}.
\end{aligned}
\end{equation}
It leads to
\begin{equation}\label{eq3.36}
\left\|u_{m}^{\prime \prime}(x,0)\right\|_{\left.L^{2}(0, L)\right)}^{2} \leqslant C\left\|\varphi(x)\right\|_{\left.H^{2}((0, L)\right)}^{2}.
\end{equation}
By substituting \eqref{eq3.33} and \eqref{eq3.36} into \eqref{eq3.32}, there exists a positive constant $C$, such that
\begin{equation}\label{eq3.37}
\begin{array}{l}
\quad\left\|u_{m}^{\prime \prime}(x, t)\right\|_{L^{2}((0, L))}^{2}+\left\|u_{m}^{\prime}(x, t)\right\|_{H_{0}^{1}((0, L))}^{2}\leqslant C\left\|\varphi(x)\right\|_{\left.H^{2}((0, L)\right)}^{2}+C\left\|\psi(x)\right\|_{ H_{0}^{1}((0, L))}^{2}.
\end{array}
\end{equation}
Similar to Theorem \ref{thm3.1}, we pass to the limit revealing that
\begin{equation}\label{eq3.38}
u^{\prime} \in L^{\infty}\left(0, T ; H_{0}^{1}((0, L))\right),  u^{\prime \prime} \in L^{\infty}\left(0, T ; L^{2}((0, L))\right).
\end{equation}
Moreover, since $u_{t t}+\sum\limits_{k=1}^{n}\beta_{k} u(x_{k}, t) \delta_{x=x_{k}}=a^{2} u_{x x}\in L^{\infty}\left(0, T ; L^{2}({(0,L)))}\right.$, according to the regularity theorem for solutions of second-order linear elliptic equations \cite{Evans1998Partial}, it follows that $u\in H^2((0,L))$. Therefore, $u \in C\left([0, T] ; H_{0}^{1}((0, L))\right) \cap L^{\infty}\left(0, T ; H^{2}((0, L)))\right.$.
\end{proof}
\par
In the process of theoretical proof, we have already noticed that for a single elastic support, the difficulty of these two methods is roughly the same. However, when we place finite elastic supports, such as $n$ elastic supports, if the problem is considered in the sense of classical solutions, it is necessary to establish $n$ systems of equations and $2n$ constraint conditions to demonstrate the well-posedness of the solution. By contrast, if the problem is considered in the sense of weak solution, we only need to add the corresponding Dirac measures on the right side of the equation to formulate the well-posedness of the weak solution as Theorem \ref{thm3.1}. As a result, The weak solution method is a better approach to a certain degree for it not only improves mathematical modeling efficiency but also simplifies the process of explaining the properties of the solution.

\section{Conclusions and discussions}\label{sec4.1}
\par
Firstly, we have reformulated the original model \eqref{eq1.1} proposed in practical engineering and obtain a model \eqref{eq2.3}-\eqref{eq2.2}. This new model enables us to establish corresponding boundary conditions for the physical phenomenon of an internally-constrained string with elastic supports, allowing us to derive the formal analytical solution and its properties in the sense of classical solution. This provides a new insight to complement the engineering approach, which typically relies on numerical approximation methods (FEM) to describe the effects brought by  internal elastic supports of the string.
\par
However, we also notice that for \eqref{eq2.3}-\eqref{eq2.2}, if the initial value $u(x,0)=\varphi(x)\notin C^2$ or if $u(x,0)=\varphi(x)\ne 0$ at $x=l$, the problem does not have a classical solution with respect to $x$ (even when time $t$ is very small). More importantly, we have discovered that this problem actually does not admit a global $C^2$ solution, indicating a substantial difference depending on whether the string is internally-constrained by elastic supports. We provide mathematical perspective for this engineering assumption. Furthermore, we raise some questions, such as whether the series \eqref{eq111} derived from the boundary conditions converges and what the regularity of the series is or whether \eqref{eq111} is a common property of general $C^2$ functions when $\varphi(x)$ meets certain conditions. Since we have not obtained a global classical solution, it is necessary to return to the original model \eqref{eq1.1} and consider the weak solution.
\par
In this paper, we define the weak solution of the model \eqref{eq1.1}, which belongs to a new class of second-order linear hyperbolic initial-boundary value problems with $\delta$ singular external force terms, in a functional form. We introduce a real bilinear functional \eqref{eq112} to handle \eqref{eq1.1}, thereby reflecting the external forces exerted by a finite number of elastic supports on the entire system. Then, by studying the equation satisfied by the functional, we establish the well-posedness of the system \eqref{eq1.1}. Finally, we find that for a finite number of elastic supports, the modeling advantages of \eqref{eq1.1} outweigh those of \eqref{eq2.3}-\eqref{eq2.2}.
\par
Meanwhile, the frequencies from formal analytical solution brought by our newly established model \eqref{eq2.3}-\eqref{eq2.2} is very close to the results from numerical simulations based on \eqref{eq1.1} in engineering. The basic PINNs structure based on \eqref{eq2.3}-\eqref{eq2.2} also demonstrates the damping effect of the elastic support. Also inspired by \cite{DMM2024}, the reason why the results of \eqref{eq2.3}-\eqref{eq2.2} in the sense of classical solution are similar to the results of \eqref{eq1.1} in the sense of weak solution can also be considered in the future.
\section*{Acknowledgments}
This work is supported by National Natural Science Foundation of China under Grants No.12071298, No.52378320 and No.52478322. Additional support is also provided by Science and Technology Commission of Shanghai Municipality No.22DZ2229014.

\end{document}